\documentclass[11pt, oneside, a4]{article}   	
\usepackage{amsmath, amsthm, amsfonts, amssymb}

\newcommand\E{{\mathbb E}}
\newcommand\Ex{{\mathbb E}}
\newcommand\F{{\mathcal F}}
\newcommand\Prob{{\mathbb P}}
\newcommand\chitilde{\tilde{\chi}}
\newcommand\Normal{{\mathcal N}}
\newcommand\n{{\mathbf n}}

\newcommand\cA{{\mathcal A}}
\newcommand\cB{{\mathcal B}}
\newcommand\cD{{\mathcal D}}

\newcommand\cM{{\mathcal M}}
\newcommand\cN{{\mathcal N}}
\newcommand\cP{{\mathcal P}}

\newcommand\Z{{\mathbb Z}}
\newcommand\R{{\mathbb R}}
\newcommand\C{{\mathbb C}}

\newcommand\dto{\overset{d}{\to }}

\newcommand\lf{\langle}
\newcommand\rt{\rangle}

\newcommand\one{{\bf 1}}

\newcommand\bra[1]{\langle #1 \rangle}
\newcommand\as{\text{ as }}

\DeclareMathOperator{\Gam}{Gamma}

\DeclareMathOperator{\Tr}{Tr}
\DeclareMathOperator{\Image}{Im}

\DeclareMathOperator{\Var}{Var}
\DeclareMathOperator{\Cov}{Cov}

\DeclareMathOperator{\interior}{int}
\DeclareMathOperator{\Sine}{Sine}

\newtheorem{theorem}{Theorem}[section]
\newtheorem{corollary}[theorem]{Corollary}
\newtheorem{lemma}[theorem]{Lemma}

\theoremstyle{definition}
\newtheorem{definition}[theorem]{Definition}

\theoremstyle{remark}
\newtheorem{remark}[theorem]{Remark}

\begin{document}

\title{Gaussian beta ensembles at high temperature: eigenvalue fluctuations and bulk statistics \thanks{This work is partially supported by JSPS KAKENHI Grant Numbers JP16K17616(T.K.D) and JP26400145(F.N.)}
}

\author{Trinh Khanh Duy
 \footnote{Institute of Mathematics for Industry, Kyushu University, Japan. Email: trinh@imi.kyushu-u.ac.jp} 
\and Fumihiko Nakano 
\footnote{Department of Mathematics, Gakushuin University, Japan. Email: fumihiko@math.gakushuin.ac.jp}}

\maketitle

\begin{abstract}
We study the limiting behavior of Gaussian beta ensembles in the regime where  $\beta n = const$ as $n \to \infty$. The results are (1) Gaussian fluctuations for linear statistics of the eigenvalues, and (2) Poisson convergence of the bulk statistics. 
(2) is an alternative proof of the result by F.~Benaych-Georges and S.~P\'ech\'e (2015) with the explicit form of the intensity measure. 

\medskip

	
\medskip
	
%\keywords{Gaussian beta ensembles \and high temperature \and random Jacobi matrices \and global fluctuations \and bulk statistics \and Poisson statistics}
 %\subclass{Primary 60B20 \and Secondary 82B44 \and 60F05 \and 60G55}
\end{abstract}

\section{Introduction}
\subsection{Background}
Gaussian beta ensembles 
(G$\beta$E) are the ensembles of points on the real line with the joint density function given by 
\begin{equation}\label{GbE}
	(\lambda_1, \dots, \lambda_n) 
	\propto 
	|\Delta(\lambda)|^\beta 
	e^{- \frac{1}{2}(\lambda_1^2 + \cdots + \lambda_n^2)} d\lambda,  
\end{equation}
where 
$\Delta(\lambda) = \prod_{i < j} (\lambda_j - \lambda_i)$ 
denotes the Vandermonde determinant. 
They are 
generalizations of the well-known Gaussian Orthogonal/Unitary/Symplectic Ensembles, and can also be viewed as the equilibrium measure of a one dimensional Coulomb log-gas at the inverse temperature 
$\beta$.

Dumitriu and Edelman \cite{DE02} introduced a matrix model whose eigenvalues obey  G$\beta$E~\eqref{GbE}. 
It is the ensemble of finite symmetric tridiagonal matrices, called Jacobi matrices, with independent entries distributed as 
\[
	T_{n, \beta} =\begin{pmatrix}
		\Normal(0,1) 	&\chitilde_{(n - 1)\beta}	\\
		\chitilde_{(n - 1)\beta}	&\Normal(0,1)		&\chitilde_{(n - 2)\beta}	\\
							&\ddots			&\ddots				&\ddots\\
							&				&\chitilde_\beta			&\Normal(0,1)
	\end{pmatrix}, 
\]
where 
$\Normal(\mu, \sigma^2)$ 
denotes the Gaussian distribution with mean 
$\mu$ 
and variance 
$\sigma^2$, 
and 
$\chitilde_k$ ($k>0$) 
denotes the 
$(1/\sqrt{2})$-chi distribution with 
$k$ 
degrees of freedom or equivalently the square root of the gamma distribution 
$\Gam(k/2,1)$. 

For fixed $\beta$, 
there are many papers on 
G$\beta$E and $T_{n, \beta}$
(e.g., 
the convergence and fluctuations around the semi-circle distribution of the empirical measures 
\cite{DE06,Johansson98}, 
the convergence and fluctuations around the semi-circle distribution of the spectral measures
\cite{Duy2016}, 
edge scaling limit 
\cite{Ramirez-Rider-Virag-2011}, 
bulk scaling limit 
\cite{Valko-Virag-2009}, 
and 
a central limit theorem (CLT for short) for the log-determinant \cite{Duy-RIMS}).

The aim 
of this paper is to study the limiting behavior of the spectra of Gaussian beta ensembles as 
$n \to \infty$ 
and 
$\beta \to 0$ 
such that 
$n \beta = const.$

\subsection{Notations}
In this subsection 
we introduce some basic notions and fix notations. 
A Jacobi matrix 
is a symmetric tridiagonal matrix with positive entries in the subdiagonal. 
In this paper 
we will deal with three types of Jacobi matrices: finite, infinite and doubly infinite matrices. 
The empirical distribution/measure 
is defined, for a finite Jacobi matrix 
$J$ 
of size 
$n$ 
with eigenvalues 
$\{ \lambda_j \}_{j=1}^n$, 
by 
\[
L_n 
:=
\frac 1n \sum_{j=1}^n \delta_{\lambda_j},
\]
where $\delta_\lambda$ denotes the Dirac measure.
Note that a finite Jacobi matrix of size $n$ has exactly $n$ distinct real eigenvalues.
The spectral measure 
may be considered for a Jacobi matrix 
$J$ 
of any type.
First of all, 
there is a measure 
$\mu$ 
satisfying
\[
	\int x^k d\mu  = (J^k e_1, e_1) = J^k(1,1), 
	\quad 
	k = 0,1,\dots.
\] 
A measure $\mu$ is unique only if it is determined by moments and is called the spectral measure of 
$J$ 
or more precisely the spectral measure of $(J, e_1)$. 
In the case of 
infinite Jacobi matrices, a sufficient condition for the uniqueness is 
\begin{equation}\label{sufficient-condition-Jacobi-matrix}
	\sum_{i = 1}^\infty \frac{1}{b_i} = \infty
\end{equation}
\cite[Corollary~3.8.9]{Simon-book-2011},  
where $\{b_i\}_{i = 1}^\infty$ denote the subdiagonal entries.
A finite Jacobi matrix 
$J$ 
always has the spectral measure which is expressed as 
\[
	\nu_n = \sum_{j = 1}^n q_j^2 \delta_{\lambda_j},\quad q_j = |u_j(1)|,
\]  
where 
$\{ u_j \}_{j=1}^n$ 
are normalized eigenvectors corresponding to the eigenvalues 
$\{ \lambda_j \}_{j=1}^n$. 
For the case of 
$T_{n, \beta}$, 
it is known that the weights 
$\{ q_j^2 \}_{j=1}^n$ 
are distributed as Dirichlet distribution with parameter 
$\beta /2$ 
and are independent of the eigenvalues 
$\{ \lambda_j \}_{j=1}^n$
\cite{DE02}.
An easy but important 
consequence of this fact is that the empirical distribution 
$L_{n, \beta}$ 
and the spectral measure 
$\nu_{n, \beta}$ 
of 
$T_{n, \beta}$ 
have the same mean, 
$\bar L_{n, \beta} = \bar \nu_{n, \beta}$, 
where the mean 
$\bar{\mu}$
of a random probability measure 
$\mu$
is defined by
\[
	\bar \mu(A) = \E[\mu(A)]
\]
for all Borel sets $A$. 
\subsection{Results}
For fixed 
$\beta$, 
it is well known that both the empirical measure 
and the spectral measure 
of 
$(
1 + n \beta/2
)^{-1/2} 
T_{n, \beta}$
converge weakly to the semi-circle distribution almost surely (Wigner's semi-circle law). 
In fact, 
the limiting behavior for  spectral measures in general follows directly from those of entries. In addition, the distance between the two measures converges to zero, which gives 
another point of view 
to the classical Wigner's semi-circle law in terms of spectral measures \cite{Duy2016}. Note that the results still hold when $\beta$ varies but $n\beta \to \infty$.

The main subject 
of this paper is to consider the joint limit such that 
$n \to \infty$ 
and 
$\beta \to 0$ with $n \beta$ being bounded.
The following 
results have been known 
\cite{Peche15,DS15}.
When  
$n \beta = 2 \alpha$, 
each entry of 
$T_{n, \beta}$ 
converges in distribution to the corresponding entry of the i.i.d.~(independent identically distributed) Jacobi matrix 
$J_\alpha$, 
where 
\[
	J_\alpha = \begin{pmatrix}
		\Normal(0,1)		&\chitilde_{2 \alpha}	\\
		\chitilde_{2 \alpha}	&\Normal(0,1)		&\chitilde_{2 \alpha}		\\
					
												&\ddots		&\ddots		&\ddots 
	\end{pmatrix}.
\]
Since the subdiagonal of 
$J_\alpha$ 
is an i.i.d.~sequence, the condition~\eqref{sufficient-condition-Jacobi-matrix} holds almost surely, so that the spectral measure 
$\mu_\alpha$ 
of 
$J_\alpha$ 
is well-defined. 
Consequently, 
the spectral measure $\nu_{n, \beta}$ of $T_{n, \beta}$ 
converges weakly to  
$\mu_\alpha$ 
in distribution, and thus, the mean 
$\bar\nu_{n, \beta} = \bar L_{n, \beta}$ 
converges weakly to 
$\bar \mu_\alpha$. 
Being different 
from spectral measures, the empirical distribution 
$L_{n, \beta}$ 
converges weakly to 
$\bar \mu_\alpha$ 
in probability. 
This is stated in  \cite{Peche15} 
and it is also possible to give an alternative proof by using the arguments in \cite{DS15}. 
That  the empirical distribution 
converges to a non random measure corresponds to the existence of the integrated density of states in the context of random Schr\"odinger operators, where its density is called the density of states.

Moreover 
the limiting measure 
$\bar\mu_\alpha$ 
is explicitly computed in 
\cite{Allez12,DS15} 
and is referred to as the probability measure of associated Hermite polynomials \cite{Askey-Wimp-1984} whose density is given by 
\[
		\bar \mu_\alpha (E) = \frac{e^{-E^2/2}}{\sqrt{2\pi}}\frac{1}{|\hat f_\alpha(E)|^2} ,
		\text{where }
		\hat f_\alpha(E) =  \sqrt{\frac{\alpha}{\Gamma(\alpha)}} \int_0^\infty t^{\alpha - 1} 
		e^{-\frac {t^2}{2} + iEt} dt.		
\]
As is remarked in \cite{Peche15}, the scaled measure
$(\alpha+1)^{1/2} 
\bar{\mu}_{\alpha}((\alpha+1)^{1/2}  E)$ 
tends to the semicircle distribution (resp.~$\Normal(0,1)$) as 
$\alpha$ 
tends to 
infinity
(resp.~zero), being consistent with the results stated in the preceding paragraph. 
Hence 
$\bar{\mu}_{\alpha}$ 
may be regarded as an interpolation between these two measures. 
A natural problem 
now is to study the fluctuation around the limit (that is, a CLT type statement). 
\begin{theorem}
\label{CLT}
Assume that a function $f$ has continuous derivative of polynomial growth. Then as $n \to \infty$ with $n \beta = 2 \alpha$,
\[
	\sqrt{n}(\lf L_{n, \beta}, f\rt - \E[\lf L_{n, \beta}, f\rt]) \dto \Normal(0, \sigma_f^2),
\]
for some constant 
$\sigma_f^2 \ge 0$. Here $\lf \mu, f\rt := \int f d\mu$, and `$\dto$' denotes the convergence in distribution.
\end{theorem}
Remark that all the results are stated and proved for $n\beta = 2\alpha$. However, all the arguments still work if $n\beta \to 2\alpha \in [0,\infty)$ because in this regime $n\beta$ stays bounded.   

There are already some results on CLT 
\cite{DE06,Johansson98} 
for fixed 
$\beta$ in which the limiting variance is given explicitly. However, these approaches do not directly apply to our problem; or at least we need more work.  
We propose here another approach which is based on the martingale difference central limit theorem to derive a CLT for polynomial test functions.
Then 
we extend the CLT to continuous functions with continuous derivative of polynomial growth by the method which has recently been developed in \cite{Duy2016}.

The next problem is to consider the bulk scaling limit, that is, to study the limiting behavior of following point process
\[
	\xi_n = \sum_{j = 1}^n \delta_{n(\lambda_j - E)},
\]
where 
$\{ \lambda_j \}_{j = 1}^n$ 
are the eigenvalues of 
$T_{n, \beta}$ and $E$ is a fixed real number.
It is proved in 
\cite{Peche15} 
that as $n \to \infty$ 
with 
$\beta n = 2 \alpha$, 
$\xi_n$ 
converges to a homogeneous Poisson point process with intensity (cf.~Eq.~(7) in \cite{Peche15})
\[
	\theta_E = \frac{1}{\sqrt{2 \pi} \Gamma(\alpha + 1)} \exp\left(-\frac{E^2}{2} + 2 \alpha \int \log |E - y| \bar\mu_\alpha(dy)  \right). 
\]
We note that the 
$\Sine_{\beta}$
process, which is the bulk scaling limit of G$\beta$E with fixed 
$\beta$,  
converges to the Poisson point process as 
$\beta \to 0$
\cite{Allez-Dumaz-2014}, 
which is consistent with the statement above. 
The approach in \cite{Peche15}
is based on analyzing  the joint density of G$\beta$E. 
It was 
conjectured that the intensity 
$\theta_E$ 
should agree with the density of states 
$\bar\mu_\alpha(E)$. 
In this paper, 
we derive the same result with desired intensity 
$\bar\mu_\alpha(E)$. 
\begin{theorem}
\label{Poisson}
As $n \to \infty$ with $n \beta = 2\alpha$, the point process $\xi_n$ converges weakly to a homogeneous Poisson point process with intensity $\bar \mu_\alpha(E)$.
\end{theorem}
For the proof we note that 
$T_{n, \beta}$ 
exhibits the Anderson localization, that is, the eigenvectors are exponentially localized, so that we can make use of the well-established method in the field of random Schr\"odinger operators \cite{Minami-1996}.
To apply the ideas in 
\cite{Minami-1996}, we need 
(i) Wegner's bound, (ii) Minami's bound, (iii) exponential decay of Green's functions, and (iv) local law. 
Contrary to usual cases, 
the main issue here is to prove 
(iv) 
because 
$T_{n, \beta}$ 
has no translation invariance. It is worth noting that a non-trivial identity $\theta_E = \bar \mu_\alpha(E)$ can be derived in an indirect way in the proof of the local law.

In the following sections 
we prove 
Theorems \ref{CLT} and \ref{Poisson}. 
Since 
our arguments are not specified to these particular matrices 
$T_{n, \beta}$, 
we work under a general setting. 
In Section~2, 
we consider random Jacobi matrices with independent entries.
By refining Minami's method, 
we prove the Poisson statistics for
$\xi_n$ 
under some mild conditions.
Two important 
sufficient conditions among them, which are non-trivial in this setting, are 
(iii) the exponential decay of Green's functions, and 
(iv) the local law. 
In Section 3, 
we show that (iii) holds for a class of Jacobi matrices  including G$\beta$E.
There are 
some approaches for that among which we use the so-called the operator method 
\cite{Schenker-2009}.
Section 4 is 
the main part of this paper where we prove Theorem \ref{CLT} and the local law. 
Since our model 
is closely related to i.i.d.~Jacobi matrices, we will also discuss some known results about i.i.d.~Jacobi matrices.  
In Appendix A (resp.~B)
we recall the martingale difference CLT
(resp.~precise definition of convergence in distribution of random probability measures). 

\section{Poisson statistics}
Consider a sequence of random Jacobi matrices 
\[
	J_n = \begin{pmatrix}
		a_1		&b_1		\\
		b_1		&a_2		&b_2\\
		&\ddots	&\ddots	&\ddots\\
		&&		b_{n - 1}	&a_n
	\end{pmatrix},
\]
where $\{a_i\}_{i = 1}^n$ and $\{b_i\}_{i = 1}^{n - 1}$ are independent random variables with an assumption that $b_i > 0, i = 1, \dots, n - 1$. For different $n$, the sequences $\{a_i\}_{i = 1}^n$ and $\{b_i\}_{i = 1}^{n - 1}$ may be different. Let $\lambda_1, \dots, \lambda_n$ be the eigenvalues of $J_n$. Let $\xi_n$ be the local statistics around $E \in \R$, that is, a point process on $\R$ defined as
\begin{equation}\label{local-statistics}
	\xi_n = \sum_{j = 1}^n \delta_{n (\lambda_j - E)}.
\end{equation}
A real number $E$ is referred to as a reference energy. The purpose of this section is to provide sufficient conditions for the point process $\xi_n$ to converge to a homogeneous Poisson point process.

When $\{a_i\}_{i = 1}^\infty$ and $\{b_i\}_{i = 1}^\infty$ are stationary sequences, it is well known that for all $E \in \R$,
\[
	\frac{1}{n} \# \{1 \le j \le n : \lambda_j \le E\} \to N(E) \text{ almost surely as } n \to \infty,
\]
where $N(E)$ is a non-random function called the integrated density of states \cite{Carmona-Lacroix}. The derivative $\n(E) = dN(E)/dE$ when exists is called the density of states at the energy $E$.

Jacobi matrices with $b_i \equiv 1$ are called discrete Schr\"odinger operators and the diagonal $\{a_i\}$ is referred to as potentials.
In case of i.i.d.~potentials, when the common distribution has bounded density and a bounded moment of positive order, the Green's function of $J_n$ decays exponentially fast. As a result, the local statistics converges to a homogeneous Poisson point process with intensity $\n(E)$, provided that $\n(E)$ exists and is positive~\cite{Minami-1996}.

This section generalizes the above well-known result on discrete Schr\"odinger operators to the case of general random Jacobi matrices. Our result can be roughly stated as follows. Under some mild conditions on $\{a_i\}$ and $\{b_i\}$, the Poisson statistics follows under the assumption that Green's functions decay exponentially fast and an additional condition called the local law. Here the local law requires that the expected number of points of the point process $\xi_n$ lying in a bounded interval is proportional to (in the limit as $n\to \infty$) the length of the interval. In the i.i.d.~case, the local law is a consequence of the exponential decay of Green's functions. See Section~\ref{sec:global-local} for more details.

Now let us explain some terminologies. Let $\cM(\R)$ be the space of all non-negative Radon measures on $\R$ equipped with the vague topology. Here the vague topology is the topology in which a sequence $\{\mu_n\} \subset \cM(\R)$ converges to $\mu \in \cM(\R)$ if 
\[
	\int_\R f(x) \mu_n(dx) =: \mu_n(f)  \to  \mu(f) \text{ as } n \to \infty,
\]
for all $f$ in $C_K^+(\R)$, the space of all non-negative continuous functions with compact support. A subset $\cN(\R)$ of all integer valued Radon measures on $\R$ then becomes a closed set in $\cM(\R)$. A point process is defined to be an $\cN(\R)$-valued random variable. Note that an element $\xi \in \cN(\R)$ can be written as 
\[
	\xi = \sum_j \delta_{x_j} ,
\]
where $\{x_j\}$ is a sequence of real numbers having no finite accumulation point.

An important example of point processes is a Poisson point process. Let $\mu$ be a Radon measure on $\R$. A point process $\xi$ is said to be a Poisson point process with intensity measure $\mu$ if it satisfies the following two conditions:
\begin{itemize}
\item[(a)]	for bounded Borel set $A$, $\xi(A)$ has Poisson distribution with parameter $\mu(A)$, namely
	\[
		\Prob(\xi(A) = k ) = e^{-\mu(A)} \frac{\mu(A)^k}{k!}, k = 0,1,\dots;
	\]
\item[(b)] for disjoint bounded Borel sets $A_1, \dots, A_m$, $\xi(A_1), \dots, \xi(A_m)$ are independent.  
\end{itemize}
A Poisson point process with intensity measure $\theta dx$ is called a homogeneous Poisson point process with intensity $\theta$.

A sequence of point processes $\{\xi_n\}_{n = 1}^\infty$ is said to converge weakly (or in distribution) to a point process $\xi$ if for any bounded continuous function $\Phi$ on $\cN(\R)$,  
\[
	\E[\Phi(\xi_n)] \to \E[ \Phi(\xi)] \text{ as } n \to \infty.
\]
Note that $\{\xi_n\}$ and $\xi$ may be defined on different probability spaces but we use the same symbol $\E$ to denote the expectation.
The weak convergence of point processes is known to be equivalent to the following statement:
for any $\varphi \in C_K^+(\R)$, 
\[
	\lim_{n \to \infty} \E [e^{- \xi_n(\varphi)}] = \E [e^{-\xi(\varphi)}].
\]

For  $\zeta = \sigma + i \tau \in \C_+:=\{z \in \C: \Image z > 0\}$, let  
\[
	f_\zeta(x) := \Image \frac{1}{x - \zeta} = \frac{\tau}{(x - \sigma)^2 + \tau^2}.
\]
Define a class $\cA$ of test functions of the form 
\[
	f(x) = \sum_{j = 1}^m \frac{\alpha_j \tau}{(x - \sigma_j)^2 + \tau^2} = \sum_{j = 1}^m \alpha_j f_{\zeta_j} (x),
\]
with $m \ge 1, \tau > 0$ and $\alpha_j > 0, \sigma_j \in \R, \zeta_j = \sigma_j + i \tau$ for $j = 1, \dots, m$. We will use the following criterion for the weak convergence of point processes.
\begin{lemma}[{\cite[Lemma~1]{Minami-1996}}]\label{lem:wc-criterion}
	Let $\{\xi_n\}_{n = 1}^\infty$ and $\xi$ be point processes such that $\E[\xi_n(dx)] \le C dx$, and $\E[\xi(dx)] \le C dx$. Then $\xi_n$ converges weakly to $\xi$, if and only if  for any $f \in \cA$, 
	\[
		\E[e^{- \xi_n(f)}] \to \E[e^{-\xi(f)}] \text{ as } n \to \infty.
	\]
Here $\E[\xi(dx)]$ denotes the intensity measure or the mean measure of a point process $\xi$, a measure $\mu$ defined as
\[
	\mu(A) = \E[\xi(A)], \text{ for all bounded Borel sets $A$.} 
\]
\end{lemma}

We are now in a position to give sufficient conditions for the Poisson statistics. For $\zeta = \sigma + i \tau$, it is straight forward to deduce that
\[
	\xi_n(f_\zeta) = \frac{1}{n} \sum_{j = 1}^n \Image \frac{1}{\lambda_j - (E + \frac{\zeta}{n})} = \frac{1}{n} \Image \Tr G_n(z),  (z = E+ \frac{\zeta}{n}),
\]
where $G_n(z) = (J_n - z)^{-1}$ is the Green's function, or the resolvent of $J_n$. Recall that our aim is to consider the limiting behavior of the local statistics $\xi_n$ associated with the Jacobi matrix $J_n$ as $n$ tends to infinity. The sequence $\{a_i\}_{i = 1}^n$ and $\{b_i\}_{i = 1}^{n - 1}$, in general, depend on $n$ but all constants in this paper will be assumed to be independent of $n$. Sufficient conditions for the Poisson statistics read as follows.
\begin{itemize}
\item[A.] The random variable $a_i$ has probability density function $\rho_i$, and $\rho_i$ is uniformly bounded, that is, 
\[
	\|\rho_i\|_\infty \le M_A, (i = 1, \dots, n).  
\]

\item[B.] 
For some $T > 1$,
\[
	\E[b_i^T] \le M_B, (i = 1, \dots, n - 1).
\]

\item[G.] (Exponential decay of Green's functions)
For some $0 < s < 1$, there are positive constants $M_s, \gamma_s$ and $\delta_s$ such that
\begin{equation}
	\E[|G_{[u,v]}(z; y,x)|^s]  \le M_s e^{- \gamma_s |y - x|}, \text{ for }y \in\{u, v\},  x \in [u, v],
\end{equation}
and for all $z \in \{z \in \C_+ : |z - E| < \delta_s\}$, and all $1\le u<v \le n$.
Here $G_{[u, v]}(z) = (J_n^{[u, v]} - z)^{-1}$ is the Green's function of $J_n^{[u, v]}=\{J_n(i,j)\}_{i, j\in [u,v]}$, the restriction of $J_n$ on $[u, v] = \{u, u+1, \dots, v\}$.

\item[L.] (Local law) There exists a positive constant $\theta$ such that
\begin{equation}\label{local-law-1}
	\E[\xi_n (I)] \to \theta |I| \as n \to \infty,  
\end{equation}
for all bounded intervals $I$. Note that $\xi_n(I) = \#\{\lambda_j \in E + \frac{I}{n}\}$.

\item[L$'$.] (Local law) There exists a positive constant $\theta$ such that for all $\zeta \in \C_+$,
\begin{equation}\label{local-law-2}
	\E[\xi_n(f_\zeta)] \to \pi \theta \as n \to \infty.
\end{equation}
\end{itemize}

When the intensity measures of the point processes $\{\xi_n\}$ are uniformly bounded, that is, $\E[\xi_n(dx)] \le C dx$, then Condition~L$'$ implies Condition~L. Indeed, assume that Condition~L$'$ holds. Let $I$ be a bounded interval. Then there are functions $\{f_k\}_{k \ge 1}$ in $\cA$,
\[
	f_k = \sum_{j:finite} \alpha_{k,j} f_{\zeta_{k,j}}, (\alpha_{k, j} > 0),
\]
that converge to $\one_{I}$ in $L^1(\R)$ \cite{Minami-1996}. For each $k$, it follows from Condition L$'$ that 
\[
	\E[\xi_n(f_k)] \to \pi \theta \sum_{j} \alpha_{k,j} = \theta \|f_k\|_{L^1(\R)} \as n \to \infty. 
\]
In addition, for all $n$, by the uniformly bounded assumption,
\[
	|\E[\xi_n(f_k)] -\E[ \xi_n (\one_I)]  |\le C\|f_k - \one_I\|_{L^1(\R)}.
\]
Therefore, 
\[
	\E[\xi_n(\one_I)]  \to \theta \|\one_I\|_{L^1(\R)} = \theta |I| \text{ as } n \to \infty. 
\]

Now we can state the main result in this section.
\begin{theorem}\label{thm:Poisson-main}
Assume that Conditions A, B, G and L hold. Then the local statistics  
\[
	\xi_n = \sum_{j = 1}^n \delta_{n (\lambda_j - E)}
\] 
converges weakly to a homogeneous Poisson point process with intensity $\theta$.
\end{theorem}

Let us give a sketch of the proof of Theorem~\ref{thm:Poisson-main}. The main stream is similar to \cite{Minami-1996}. We will omit proofs of trivially  extended results. The idea is as follows. Divide $[1,n]$ into small intervals $C_1, \dots, C_m$ of length $\sim n^\alpha$, $0 < \alpha < 1$. For each $p$, consider the restriction of $J_n$ on $C_p$ and the point process 
\[
	\eta_{n, p} = \sum_j \delta_{n(\lambda_j^{(C_p)} - E)},
\]
where $\{\lambda_j^{(C_p)}\}$ are the eigenvalues of $J_n^{C_p}$. Then $\xi_n$ is well approximated by the sum of independent negligible point processes $\{\eta_{n, p}\}_p$, which implies the convergence to a Poisson point process.

In order to apply the criterion for the weak convergence of point processes stated in Lemma~\ref{lem:wc-criterion}, we need the following result which is well known as Wegner's estimate. See \cite{Minami-1996} and references therein for the proof.
\begin{lemma}[Wegner's estimate]\label{lem:Wegner}
Assume that Condition A holds. Then 	
\[
		\E[\Image G_{[u, v]} (z; x,x)] \le M_A \pi,
\]
for all $z \in \C_+$, and $1\le u \le x \le v \le n$. Consequently, $\E[\xi_n(f_\zeta)] \le M_A \pi$ for all $\zeta \in \C_+$, and hence, $\E[\xi_n(dx)] \le M_A dx$.
\end{lemma}

The following result shows that $\xi_n$ is well approximated by the sum of $\{\eta_{n, p}\}_p$.
\begin{lemma}\label{lem:approximate-xi-n}
Assume that Conditions A, B and G hold. Then for all $\zeta \in \C_+$,
\[
	\sum_p \eta_{n, p} (f_\zeta) - \xi_n(f_\zeta) \to 0 \text{ as $n\to \infty$ in $L^1$ and in probability.}
\]
\end{lemma}
\begin{proof}
	This is a generalization of Step~3 in \cite{Minami-1996} to the case of Jacobi matrices.	We begin with the following expression 
	\begin{equation}\label{explicit-formula}
		\sum_p \eta_{n, p} (f_\zeta) - \xi_n(f_\zeta)  = \frac{1}{n} \sum_p\sum_{x\in C_p} \left( \Image G_{C_p}(z; x, x) - \Image G_n(z; x,x) \right), 
	\end{equation}
where $z = E + \frac{\zeta}{n}$. For simplicity of notations, let $C_p = [u,v]$. It follows from the resolvent equation that
\begin{align*}
	&G_{C_p}(z; x, x) - G_n(z; x, x) \nonumber\\
	&= G_n(z; x, u - 1) b_{u - 1} G_{[u,v]}(z; u, x) +  G_n(z; x, v + 1) b_{v} G_{[u,v]}(z; v, x). \label{resolvent-eq-Cp}
\end{align*}
Here by setting $b_0 := 0$ and $b_n := 0$, the first term or the second term vanishes when $u = 1$ or $v = n$. We bound the first term as follows 
\begin{align*}
	&\E[|G_n(z; x, u - 1) b_{u - 1} G_{[u,v]}(z; u, x)|] \\
	&\le \frac{1}{|\Image z|^{2 - \varepsilon}} \E[b_{u - 1} |G_{[u, v]}(z; u, x)|^\varepsilon] \\
	&\le \frac{1}{|\Image z|^{2 - \varepsilon}}  \E[b_{u - 1}^T]^{1/T} \E[|G_{[u,v]}(z; u, x) |^{\varepsilon q}]^{1/q} \\
	&= \frac{1}{|\Image z|^{2 - \varepsilon}}  \E[b_{u - 1}^T]^{1/T} \E[|G_{[u,v]}(z; u, x) |^{s}]^{1/q}\\ 
	&\le \frac{1}{|\Image z|^{2 - \varepsilon}}  M_B^{1/T} (M_s e^{-\gamma_s(x - u)})^{1/q}\\ 
	&= \frac{1}{|\Image z|^{2 - \varepsilon}} \tilde M e^{-\tilde \gamma (x - u)}.
\end{align*}
Here $q$ is the H\"older conjugate number of $T$, that is, $T^{-1} + q^{-1} = 1$, $\varepsilon = s/q$, and $\tilde M$ and $\tilde \gamma$ are positive constants. We have used H\"older's inequality and trivial estimates that $|G_n(z; x, u - 1| \le 1/\Image z$, and $|G_{[u, v]}(z; u, x)| \le 1/\Image z$. Thus, for $z = E + \zeta / n$, 
\[
	\E[|G_n(z; x, u - 1) b_{u - 1} G_{[u,v]}(z; u, x)|] \le  \hat M n^{2 - \varepsilon} e^{-\tilde \gamma (x - u)} \le \hat M n^{-\varepsilon},
\]
if $x - u > 2 \log n / \tilde \gamma$. Now let 
\begin{align*}
	\interior(C_p) &= \interior([u,v]) := (u + 2 \log n / \tilde \gamma, v - 2 \log n/\tilde \gamma),\\
	\partial(C_p) &= C_p \setminus \interior (C_p).
\end{align*}
With these notations, we see that
\[
	\frac{1}{n}\sum_p \sum_{x \in \interior (C_p)} \E[|G_n(z; x, u - 1) b_{u - 1} G_{[u,v]}(z; u, x)|]  \to 0 \text{ as } n \to \infty.
\]
The second term can be estimated in the same way. Consequently,  
\[
	\frac{1}{n}\sum_p \sum_{x \in \interior(C_p)} |\E[\Image G_{C_p}(z; x, x)] - \E[\Image G_{n}(z; x,x)]| \to 0 \text{ as } n \to \infty.  
\]

For $x \in \partial(C_p)$, note that the expectation of each summand in \eqref{explicit-formula} is bounded by $2 M_A \pi$ by Lemma~\ref{lem:Wegner}. Thus 
\begin{align*}
	&\frac{1}{n}\sum_p \sum_{x \in \partial(C_p)} |\E[\Image G_{C_p}(z; x, x)] - \E[\Image G_{n}(z; x,x)]| \\
	& \le \frac{2 M_A \pi}{n} \sum_p {\# \partial(C_p)} \to 0 \text{ as } n \to \infty.
\end{align*}	
The proof of Lemma~\ref{lem:approximate-xi-n} is complete. 
\end{proof}

Let $\eta_n = \sum_p \eta_{n, p}$. Then $\E[\eta_n (dx)] \le \pi M_A dx$ by Lemma~\ref{lem:Wegner}. Moreover, Lemma~\ref{lem:approximate-xi-n} implies that 
\[
	\xi_n(f) - \eta_n(f) \to 0 \text{ in $L^1$ and in probability,}
\]
for all $f \in \cA$. Thus, $\xi_n$ and $\eta_n$ have the same limit by taking into account Lemma~\ref{lem:wc-criterion}.

\begin{corollary}
Assume that Conditions A, B, G and L hold. Then for any bounded interval $I$,
\[
	\xi_n(I) - \eta_n (I) \to 0 \text{ in $L^1$ as $n \to \infty$.}
\]
Consequently, 
\begin{equation}\label{Poisson-0}
	\E[\eta_n (I)] = \sum_p \E[\eta_{n, p}(I)] \to \theta |I| \as n \to \infty.
\end{equation}

\end{corollary}
\begin{proof}
The proof is omitted because it is similar to the previous proof of deriving Condition~L from Condition L$'$.

\end{proof}

Finally, the negligibility of the point processes $\{\eta_{n, p}\}_p$ is governed by Minami's estimate. The following statement is a trivial extension of the equation (2.53) in \cite{Minami-1996} to the case of general Jacobi matrices.
\begin{lemma}[Minami's estimate]
Assume that Condition A holds. Then for any bounded interval $I$,
	\begin{equation}\label{Poisson-2}
		\sum_{p}\sum_{j \ge 2} \Prob(\eta_{n,p}(I) \ge j) \to 0 \as n \to \infty.
	\end{equation}
\end{lemma}

Minami's estimate, together with the equation~\eqref{Poisson-0}, yields 
\begin{equation}\label{Poisson-1}
	\sum_p \Prob(\eta_{n,p} (I) \ge 1) \to \theta |I| \as n \to \infty.
\end{equation}
Therefore, $\eta_n$, as the sum of independent negligible point processes $\{\eta_{n,p}\}_p$, converges weakly to a Poisson point process with intensity $\theta$ by \cite[Theorem~9.2.V]{Daley-Vere-Jones-1988}. Consequently, the point process $\xi_n$ also converges weakly to that Poisson point process because $\{\xi_n\}$ and $\{\eta_n\}$ have the same limit. The proof of Theorem~\ref{thm:Poisson-main} is complete.

\section{Exponential decay of Green's functions}\label{sec:G}
In this section, we consider the case when $\{a_i\}_{i = 1}^\infty$ is an i.i.d.~sequence of random variables. The sequences $\{\{b_i\}_{i = 1}^{n-1} \}_n$ may depend on $n$. We show that under Conditions A, B and an additional condition on the  regularity of the common probability density functions of $\{a_i\}$, the exponential decay of  Green's functions holds, that is, Condition G automatically holds.

A random variable with probability density function $\rho$ is said to be fluctuation regular if there are positive constants $\varepsilon, \delta$, and a measurable set $R \subset \R$ with $\int_R \rho(a) da > 0$ such that for any $a \in R$, and all $x_1, x_2 \in (a - \varepsilon, a + \varepsilon)$, 
\[
	\frac{\rho(x_1)}{\rho(x_2)} \ge \delta.
\]

\begin{theorem}[cf.~{\cite[Theorem~4]{Schenker-2009}}] \label{thm:Schenker-4}
Assume that Conditions A and B hold. Moreover, assume that the common probability density function of  the i.i.d.~sequence $\{a_i\}_{i = 1}^\infty$ is fluctuation regular. Then for $0 < s < 1$ and $\Lambda > 0$, there are positive constants $M$ and $\gamma$ such that for all $\lambda \in [- \Lambda, \Lambda]$, 
\[
	\E[|G_n(\lambda; x, y)|^s] \le M e^{-\gamma |x - y|}.
\]
\end{theorem}
\begin{proof}
The proof follows the same lines as that of Theorem~4 in \cite{Schenker-2009}. The only thing we need to verify is the uniform estimate in (2.46). But it is an easy consequence of Condition~B because for any $b > 0$, 
\[
	\Prob(b_i > b) \le \frac{1}{b^T} \E[b_i ^T] \le \frac{M_B}{b^T}. \qedhere
\]
\end{proof}

\begin{theorem} \label{thm:exponential-decay}
Under the same assumptions as in Theorem~{\rm\ref{thm:Schenker-4}}, for $0 < s < \frac{1}{2}$, 
\[
	\E[|G_n(z; y, x)|^{s/2}] \le  M_s e^{-\frac \gamma 2 |y - x|}, \text{ for } y \in \{1, n\}, x \in [1, n], 
\]
and for all $z = \lambda + i \tau, \lambda \in [-\Lambda, \Lambda]$.
Here $\gamma$ is the constant in Theorem~{\rm\ref{thm:Schenker-4}}.	
\end{theorem}
\begin{remark}
	Since all constants here do not depend on $n$, the result holds for any restriction of $J_n$ on an interval $[u,v]$, that is,
	\[
	\E[|G_{[u, v]}(z; y, x)|^{s/2}] \le M_s e^{-\frac \gamma 2 |y - x|}, \text{ for } y \in \{u, v\}, x \in [u, v], 
\]
and for all $z = \lambda + i \tau, \lambda \in [-\Lambda, \Lambda]$, which implies Condition G.
\end{remark}

\begin{lemma}[cf.~{\cite[Lemma~5]{Graf-1994}}] \label{lem:bounded} Assume that Condition A holds. Then 
for $0<s<1$,  
\[
	\E[|G_n(z; x, y)|^s] \le C_s, \text{ for } x, y \in [1, n], \text{and for all }z \in \C,
\]
where $C_s$ is a constant which depends only on $s$ and $M_A$.
\end{lemma}

\begin{proof}[Proof of Theorem~{\rm\ref{thm:exponential-decay}}]
Without loss of generality, assume that $y = 1$. Let us first consider the case $x = n$. For Jacobi matrices, we can easily check the following relation  
\[
	G_n(z; 1, n) =  - \frac{b_1 \cdots b_{n - 1}}{\det(z - J_n)}.
\]
Note that all eigenvalues $\lambda_1, \dots, \lambda_n$ of $J_n$ are real. Thus for any $\tau \in \R$,
	\begin{align*}
		|G_n(\lambda + i \tau; 1, n)| &= \frac{b_1 \cdots b_{n - 1}}{\Big| \prod_{j = 1}^n (\lambda_j - \lambda - i\tau) \Big|} = \frac{b_1 \cdots b_{n - 1}}{\Big| \prod_{j = 1}^n \big((\lambda_j - \lambda)^2 + \tau^2 \big) \Big|^{1/2}} \\
		&\le \frac{b_1 \cdots b_{n - 1}}{\Big| \prod_{j = 1}^n (\lambda_j - \lambda)^2  \Big|^{1/2}} = |G_n(\lambda; 1, n)|.
	\end{align*}
Consequently, 
\begin{equation*}\label{image-real}
\E[|G_n(\lambda + i \tau; 1, n)|^s] \le \E[|G_n(\lambda; 1, n)|^s],
\end{equation*}
and hence,
\begin{equation}\label{Gx-estimate}
	\E[|G_n(\lambda + i \tau; 1, n)|^s]  \le \E[|G_n(\lambda; 1, n)|^s] \le M e^{-\gamma (n - 1)},
\end{equation}
for $\lambda \in [-\Lambda, \Lambda]$, where $M$ and $\gamma$ are the constants in Theorem~\ref{thm:Schenker-4}. 

Next we consider the case $x < n$. Note that the above estimate still holds, if $n$ is replaced by $x$, namely, 
\[
	\E[|G_x(\lambda + i \tau; 1, x)|^s]  \le \E[|G_x(\lambda; 1, x)|^s] \le M e^{-\gamma (x - 1)},
\]
where $G_x(z)$ denotes the Green's function of $J_x$, the restriction of $J_n$ on $[1,x]$. Then the desired bound for $\E[|G_n(\lambda + i \tau; 1, x)|^s]$ follows by using the resolvent equation and H\"older's inequality. Indeed, the resolvent equation yields  
\[
	G_{n}(z; 1, x) = G_{x}(z; 1, x) (1 + b_x G_{n}(z; x+1, 1)). 
\]
Then by H\"older's inequality,
\[
	\E[|G_{n}(z; 1, x) |^{s/2}] \le \E[|G_{x}(z; 1, x)|^{s}]^{1/2} \E[|1 + b_x G_{n}(z; x+1, 1)|^{s}]^{1/2} .
\]
In addition, the second factor  is uniformly bounded, because 
\[
	\E[|1 + b_x G_{n}(z; x+1, 1)|^{s}]  \le (1 + \E[b_x^{2s}]^{1/2} \E[|G_{n}(z; x+1, 1)|^{2s}] ^{1/2} ).
\]
Note that we need the assumption that $s < 1/2$ here. Therefore, for all $\lambda \in [-\Lambda, \Lambda]$,  
\[
	\E[|G_{n}(z; 1, x) |^{s/2}] \le M_s e^{-\frac{\gamma}{2} (x - 1)},
\]
for some constant $M_s > 0$.
Theorem~\ref{thm:exponential-decay} is proved. 
\end{proof}

\section{Global law and local law}\label{sec:global-local}

\subsection{i.i.d. Jacobi matrices}
Let $\{a_i\}_{i \in \Z}$ be an i.i.d.~sequence of random variables and $\{b_i\}_{i \in \Z}$ be another i.i.d.~sequence positive random variables which is independent of $\{a_i\}$. Let $J$ be a doubly infinite Jacobi matrix formed from $\{a_i\}$ and $\{b_i\}$,
\begin{equation}\label{doubly-infinite-matrix}
	J = \begin{pmatrix}
		\ddots	&\ddots	&\ddots\\
		&b_0		&a_1		&b_1			\\
		&&b_1		&a_2		&b_2\\
		&&&\ddots	&\ddots	&\ddots\\
	\end{pmatrix}.
\end{equation}
The Jacobi matrix $J$ is regarded as an operator on $\ell^2(\Z)$ with a domain
\[
	\cD_0 = \{\psi = (\psi_i)_{i \in \Z} \in \ell^2(\Z) : \psi_i = 0 \text{ for all but finitely many $i\in \Z$}\}.
\]
Then $J$ is essentially self-adjoint almost surely because 
\[
	\sum_{i = 1}^\infty \frac{1}{b_i^2} = \sum_{i = -\infty}^{-1}\frac{1}{b_i^2} = \infty  \text{ (almost surely)},
\]
see \cite[p.~122]{Carmona-Lacroix}. Let $G(z)$ be the resolvent of $J$, $G(z) = (J - z)^{-1}, z \in \C_+$. Then there is a unique probability measure $\mu$ on $\R$, called the spectral measure of $(J, e_1)$, satisfying  
\[
	\int_\R \frac{\mu(dx)}{x - z} = (G(z) e_1, e_1) = G(z; 1,1), z \in \C_+.
\]
The left hand side of the above formula is the Stieltjes transform of $\mu$ which is denoted by $S_\mu(z)$ from now on.
Let $\bar \mu$ be the mean of $\mu$. Then 
\[
	S_{\bar \mu} (z) = \E[S_\mu(z)] = \E[G(z; 1,1)].
\]

Consider a sequence of finite Jacobi matrices
\[
	J_{n} = \begin{pmatrix}
		a_1		& b_1	\\
		b_1		&a_2			&b_2		\\
				&\ddots		&\ddots		&\ddots \\
				&			&b_{n - 1}			& a_n			
	\end{pmatrix}.
\] 
Let 
\[
	L_n = \frac{1}{n} \sum_{j = 1}^n \delta_{\lambda_j}
\]
be the empirical distribution of $J_n$, where $\lambda_1, \dots, \lambda_n$ are the eigenvalues of $J_n$. Then 
\[
	N_n(E) = \frac{1}{n} \# \{1 \le j \le n: \lambda_j \le E\}
\]
is nothing but the distribution function of $L_n$. The following result is well known as the existence of the integrated density of states (ids for short). 
\begin{theorem}\label{thm:ids}
	The empirical distribution $L_n$ converges weakly to $\bar \mu$ almost surely as $n$ tends to infinity. This means that for any bounded continuous function $f$, 
	\[
		\lf L_n, f\rt \to \lf \bar \mu, f\rt \text{ almost surely as } n \to \infty. 
	\]
Here recall that $\bra{\mu, f} := \int f d\mu$ for a probability measure $\mu$ and a measurable function $f$.
\end{theorem}

\begin{remark}
The existence of ids can be rewritten in the following form
\[
	N_n(E) \to N(E) \text{ almost surely as } n \to \infty,
\]
at any continuous point $E$ of $N(E)$,
where $N(E)$ is the distribution function of $\bar \mu$, $N(E) = \bar \mu((-\infty, E])$.
These results may be regarded as the strong law of large numbers. Then the next natural question is about the central limit theorem (CLT) which is related to the second order of the above convergence. The following results were known.
\begin{itemize}
	\item[(i)] 
	Reznikova \cite{Reznikova-1980} considered discrete Schr\"odinger operators with i.i.d.~potentials whose common distribution has continuous probability density function with compact support. It was proved that the random process 
	\[
		N_n^*(E) := \sqrt{n}(N_n(E) - N(E))
	\]
converges to a Gaussian process in the sense of convergence of finite distributions. 
	\item[(ii)]
	Girko and Vasil$'$ev \cite{Girko-1983} considered general i.i.d.~Jacobi matrices and derived a CLT for a suitable scaling of $(\tilde N_n(E_1) - \tilde N_n(E_2))$, where
\[
	\tilde N_n(E) = \frac{1}{a} \int \frac{N_n(E + ay)}{1 + y^2} dy, (a>0),
\]
is a smooth version of $N_n(E)$.

	\item[(iii)]
Recently,	Krisch and Pastur \cite{Krish-Pastur-2015} considered discrete Schr\"odinger operators with bounded i.i.d.~potentials and established a CLT for $\Tr G_n(x)$, where $x \in \R$ does not lie in the spectrum of $J$.
	\item[(iv)] 	
	In random matrix theory, we want to extend such CLT  for as large as possible class of test functions. In the next subsection, for Gaussian beta ensembles, we are going to establish a CLT for continuous test function with continuous derivative of polynomial growth. As a preparation for the next subsection, we study a CLT for polynomial test functions in case where $\{a_i\}$ and $\{b_i\}$ have all finite moments.
\end{itemize}
\end{remark}

\begin{theorem}\label{thm:CLT-iid}
	Assume that $\{a_i\}$ and $\{b_i\}$ have all finite moments. Then for any non trivial polynomial $p$, 
	\begin{equation}\label{iid-polynomial-clt}
		\sqrt{n} (\lf L_n, p \rt - \E[\lf L_n, p \rt]) \dto \Normal(0, \sigma_p^2),
	\end{equation}
for some constant $\sigma_p^2 \ge 0$. 
\end{theorem}

\begin{proof}
Let $p$ be a polynomial of degree $m > 0$. Let us first prove the law of large numbers for $\langle L_n, p \rangle$. We begin with the following expression
\begin{align*}
	\langle L_n, p \rangle = \frac{1}{n} \Tr p(J_{n}) = \frac{1}{n} \sum_{j = 1}^n p(J_{n})(j,j).
\end{align*}
Observe that the sequence $\{p(J_n)(j,j)\}_{j = 1}^n$ is a part of a stationary process except some first and some last terms. More precisely, let $\theta_j = \theta_j (p) = p(J)(j,j)$. Recall that $J$ is the doubly infinite Jacobi matrix defined in \eqref{doubly-infinite-matrix}. Then $\{\theta_j\}_{j \in \Z}$ is a stationary process. Moreover, $p(J_n)(j,j) = \theta_j$, if $m/2 < j < n - m/2$. Now the expectation of $\theta_1$ is finite because all moments of $a_i$ and $b_i$ are finite. Thus by the ergodic theorem,
\[
	\frac{1}{n} \sum_{j = 1}^n \theta_j\to \E[\theta_1] \text{ almost surely as } n\to \infty.	
\]
Consequently, 
\[
	\langle L_n, p \rangle = \frac{1}{n} \sum_{j = 1}^n p(J_{n})(j,j) \to \E[\theta_1(p)] \text{ almost surely as } n\to \infty.
\]

Next, we consider the central limit theorem for $\bra{L_n, p}$. The idea here is to apply the martingale difference central limit theorem quoted in Appendix~\ref{app:CLT}. Let $\F_{n,k} = \sigma(a_i, b_i : 1 \le i \le k),$ for $1\le k \le n$,  $\F_{n, 0} = \{\emptyset, \Omega\}$, and let
\begin{align*}
	X_{n,k} &= \E[n\langle L_n, p \rangle | \F_{n,k}], (0 \le k \le n),\\
	Y_{n, k} &= X_{n, k} - X_{n, k - 1}, (1 \le k \le n),\\
	\sigma_{n, k}^2 &= \E[Y_{n, k}^2 | \F_{k - 1}], (1 \le k \le n).
\end{align*}
Then the CLT \eqref{iid-polynomial-clt} follows from Theorem~\ref{thm:MTG}, provided that the following two conditions are satisfied
\begin{align}
	&\frac{1}{n} \sum_{k = 1}^n \sigma_{n, k}^2 \to \sigma_p^2 \text{ in probability as } n \to \infty, 	\label{MTG1}\\
	&\frac{1}{n^2} \sum_{k = 1}^n \E[Y_{n,k}^4] \to 0 \text{ as } n \to \infty,  \label{MTG2}
\end{align}
where $\sigma_p^2 \ge 0$ is a constant.

Let us show the condition~\eqref{MTG1}. Note that $p(J_n)(j,j)$ depends on $\{a_{j \pm k}, b_{j \pm k}\}_{k = 0,\dots, \lfloor \frac m2 \rfloor}$. Therefore $Y_{n,k}$, and hence $\sigma_{n,k}^2$, depends only on $\{a_{k \pm j}, b_{k \pm j}\}_{j = 0, \dots, m}$. Consequently, $\sigma_{n,k}^2 =: \hat \sigma_k^2$ does not depend on $n$, if $m < k < n- m$. The sequence $\{\hat \sigma_k^2\}_{k > m}$ becomes a stationary process, and thus by the ergodic theorem, 
\[
	\frac{1}{n} \sum_{k = m + 1}^ {n - m - 1} \hat \sigma_k^2 \to \E[\hat \sigma_{m + 1}^2] \text{ almost surely as } n \to \infty,
\]
from which the condition \eqref{MTG1} follows.

For the condition~\eqref{MTG2}, note that $\E[|Y_{n,k}|^4] \le M$ for some constant $M$ which does not depend on $n$ and $k$. Therefore,
\[
	\frac{1}{n^2} \sum_{k = 1}^n \E[|Y_{n,k}|^4]  \le \frac Mn \to 0 \text{ as } n \to \infty.
\]
The theorem is proved. 
\end{proof}

Let us move on to the main topic of this subsection. Recall that the density $\bar \mu(E)$ (with respect to the Lebesgue measure) when exists is called the density of states. The local law in this case is a consequence of the exponential decay of Green's functions.  
\begin{lemma}\label{lem:local-law-iid}
	Assume that Conditions A, B and G holds. Assume further that $\bar \mu(E)$ exists at $E$ and is positive. Then 
\[
	\E[\xi_n(f_\zeta)] \to \pi \bar \mu(E) \as n \to \infty. 
\]
\end{lemma}
\begin{proof}
Fix $\zeta = \sigma + i \tau$. Recall that 
\[
	\xi_n(f_\zeta) = \frac{1}{n} \sum_{x = 1}^n \Image G_n(z; x, x),
\]
where $z = E + \zeta/n$. Similar to the proof of Lemma~\ref{lem:approximate-xi-n}, we can deduce that
\begin{align*}%\label{infinite-approximate}
	&\E[\xi_n(f_\zeta)] - \E[\Image G(z; 1, 1)] \\
	&= \frac{1}{n} \sum_{x = 1}^n (\E[\Image G_n(z; x, x)] - \E[\Image G(z; x, x)]) \to 0 \text{ as } n \to \infty. 
\end{align*}

In addition, note that 
\[
	\E[\Image G(z; 1, 1)] = \Image S_{\bar \mu}(E + \frac{\zeta}{n}) \to \pi \bar\mu(E),
\]
provided that the density $\bar \mu(E)$ exists at $E$, which completes the proof of Lemma~\ref{lem:local-law-iid}. 
\end{proof}

\begin{corollary}
Assume that Conditions A and B hold and that the probability density function of the common distribution of $\{a_i\}$ is fluctuation regular. Then the point processes $\{\xi_n\}$ converge weakly to a homogeneous Poisson point process with intensity $\bar \mu(E)$, provided that $\bar \mu(E)$ exists at $E$ and is positive.
\end{corollary}

\subsection{Gaussian beta ensembles at high temperature}
\subsubsection{Global law}

We consider the asymptotic behavior of the (scaled) G$\beta$E as $n \to \infty$ with $n \beta =2 \alpha$. Recall that the scaled G$\beta$E here is associated with the following random Jacobi matrix
\[
	T_{n,\beta} = \begin{pmatrix}
		\Normal(0,1)		&\chitilde_{(n-1)\beta}	\\
		\chitilde_{(n-1)\beta}	&\Normal(0,1)		&\chitilde_{(n-2)\beta}		\\
					
												&\ddots		&\ddots		&\ddots \\
						
						&&					\chitilde_{\beta}	&\Normal(0,1)
	\end{pmatrix}.
\]
We still denote the entries of $T_{n,\beta}$ by $\{a_1, \dots, a_n\}$ and $\{b_1, \dots, b_{n - 1}\}$ though they vary as functions of $n$ and $\beta$. We use the same notations as in the previous subsection. Let $L_{n, \beta} $ be the empirical distribution of $T_{n,\beta}$,
\[
	L_{n, \beta} = \frac{1}{n} \sum_{j = 1}^n \delta_{\lambda_j},
\]
where $\lambda_1, \dots, \lambda_n$ are the eigenvalues of $T_{n,\beta}$, which are distributed as 
\[
	(\lambda_1, \dots, \lambda_n) \propto |\Delta(\lambda)|^\beta e^{- \frac{1}{2}(\lambda_1^2 + \cdots +\lambda_n^2 )} d\lambda.
\]

Another interesting object when studying global limiting behaviors of Gaussian beta ensembles is the spectral measure. The spectral measure of $T_{n,\beta}$ is defined as a probability measure $\nu_{n, \beta}$ satisfying 
\[
	\bra{\nu_{n, \beta}, x^k} = (T_{n,\beta}^k e_1, e_1) = T_{n,\beta}^k(1,1), k=0,1,\dots.
\]
Let $u_1, \dots, u_n$ be normalized eigenvectors corresponding to the eigenvalues $\lambda_1, \dots, \lambda_n$. Then $u_1, \dots, u_n$ form an orthonormal system in $\R^n$ because the eigenvalues are distinct. The spectral measure $\nu_n$ can be expressed as
\[
	\nu_{n, \beta} = \sum_{j = 1}^n q_j^2 \delta_{\lambda_j},\quad q_j = |u_j(1)|.
\]
The weights $(q_1^2, \dots, q_n^2)$ are known to have symmetric Dirichlet distribution with parameter $\beta/2$,
and are independent of the eigenvalues $(\lambda_1, \dots, \lambda_n)$.

Since $q_j$ plays equal role for $j = 1, \dots, n$, it follows that $\E[q_j^2] = 1/n$. Consequently, the mean of the empirical measure coincides with the mean of the spectral measure, namely, $\bar L_{n, \beta} = \bar \nu_{n, \beta}$. In the regime that $n \beta = 2 \alpha$, it was shown in \cite{Allez12,DS15} that the mean measures $\bar L_{n, \beta} = \bar \nu_{n, \beta}$ converge weakly to a non random probability measure. The limiting measure, denoted by  $\bar \mu_\alpha$, is the spectral measure of the following infinite Jacobi matrix
\[
	\begin{pmatrix}
		0	&\sqrt{\alpha + 1}\\
		\sqrt{\alpha + 1}	&0	&\sqrt{\alpha + 2}\\
		&\ddots	&\ddots	&\ddots
	\end{pmatrix}.
\]
Thus, we call it the probability measure of associated Hermite polynomials \cite{Askey-Wimp-1984}.
The density of $\bar \mu_\alpha$ is given by
\[
		\bar \mu_\alpha (E) = \frac{e^{-E^2/2}}{\sqrt{2\pi}}\frac{1}{|\hat f_\alpha(E)|^2} , \text{where }
	\hat f_\alpha(E) =  \sqrt{\frac{\alpha}{\Gamma(\alpha)}} \int_0^\infty t^{\alpha - 1} e^{-\frac {t^2}{2} + iEt} dt.
\]

Let us shortly explain the above fact. Let $J_\alpha$ be an infinite i.i.d.~Jacobi matrix whose entries are distributed as
\[
	J_\alpha = \begin{pmatrix}
		\Normal(0,1)		&\chitilde_{2 \alpha}	\\
		\chitilde_{2 \alpha}	&\Normal(0,1)		&\chitilde_{2 \alpha}		\\
					
												&\ddots		&\ddots		&\ddots 
	\end{pmatrix}.
\]
Then each entry of $T_{n,\beta}$ converges in distribution to the corresponding entry of $J_\alpha$ as $n \to \infty$. We can choose a realization such that the convergence holds almost surely. Consequently, on that realization, all moments of the spectral measure $\nu_{n,\beta}$ converge to those of $\mu_\alpha$, the spectral measure of $J_\alpha$, almost surely. Note that the spectral measure $\mu_\alpha$ is unique, or is determined by moments, almost surely. Therefore, on such realization, $\nu_{n, \beta}$ converges weakly to $\mu_\alpha$ almost surely by Corollary~\ref{cor:moments-convergence}. In general, we can only state that  the spectral measure $\nu_{n, \beta}$ converges weakly to $\mu_\alpha$ in distribution. It follows that the mean $\bar \nu_{n, \beta}$ converges weakly to $\bar \mu_\alpha$, the mean of $\mu_\alpha$.

\begin{theorem}[Global law {\cite[Proposition 2.1]{Peche15}}]
As $n \to \infty$ with $n\beta = 2 \alpha$,
the empirical distributions $\{L_{n, \beta} \}$ converge weakly to $\bar \mu_\alpha$ in probability.
\end{theorem}

It was shown in \cite{Peche15} that $\Var[\lf L_{n, \beta}, p \rt] = O(1/n)$ as $n \to \infty$, which implies that  
\[	
	\lf L_{n, \beta}, p \rt \to \lf \bar \mu_\alpha, p\rt \text{ in probability as } n \to \infty,
\]
for any polynomial $p$. Thus the empirical distributions $\{L_{n, \beta}\}$ converge weakly to $\bar \mu_\alpha$ in probability because the probability measure $\bar \mu_\alpha$ is determined by moments (Corollary~\ref{cor:moments-convergence}). Moreover, the convergence of moments also implies that 
\[
	\lf L_{n, \beta}, f \rt \to \lf \bar \mu_\alpha, f\rt \text{ in probability as } n \to \infty,
\]
for any continuous function $f$ of polynomial growth.

Next, we investigate the fluctuation around the limit, or a type of central limit theorem for $\lf L_{n, \beta}, f \rt$.  For fixed $\beta$, several approaches could solve such problem \cite{DE06,Johansson98}. However, those methods do not seem to directly work in this case. We propose here an approach which uses the martingale difference central limit theorem to derive the CLT for polynomial test functions as in the previous subsection. Then extending the CLT to continuous test functions having continuous derivative of polynomial growth can be done by the method which has recently been developed in \cite{Duy2016}.

Note that for the i.i.d.~Jacobi matrix $J_\alpha$, the following central limit theorem has been established in the previous subsection,  
\[
	\sqrt{n} (\langle L_n(\alpha), p \rangle - \E[\langle L_n(\alpha), p \rangle] ) \dto \Normal(0, \sigma^2_p(\alpha)) \as n \to \infty. 
\]
Here $L_n(\alpha)$ is the empirical distribution of $J_{\alpha}^{[1,n]}$, the restriction of $J_{\alpha}$ on $[1,n]$. 
\begin{theorem}\label{thm:poly}
	For any non constant polynomial $p$,  as $n \to \infty$ with $n \beta = 2\alpha$,
	\[
		\sqrt{n}(\langle L_{n, \beta}, p \rangle - \E[\langle L_{n, \beta}, p \rangle]) \dto \Normal(0, \sigma^2_p),
	\]
	where $\sigma^2_p = \int_0^1 \sigma_p^2(u \alpha) du$.
\end{theorem}

\begin{proof}
The proof is similar to that of Theorem~\ref{thm:CLT-iid}. Let
\begin{align*}
	\F_{n,k} &= \sigma(a_i, b_i : 1 \le i \le k), (1 \le k \le n), \F_{n, 0} = \{\emptyset, \Omega\},\\
	X_{n,k} &= \E[n \bra{L_{n, \beta}, p} | \F_{n,k}], (0 \le k \le n),\\
	Y_{n, k} &= X_{n, k} - X_{n, k - 1},(1 \le k \le n),\\
	\sigma_{n, k}^2 &= \E[Y_{n, k}^2 | \F_{k - 1}], (1 \le k \le n).
\end{align*}
We need to check the following two conditions
\[
	\frac{1}{n} \sum_{k = 1}^n \sigma^2_{n,k}  \to \sigma_p^2 \text{ in probability as $n \to \infty$}, 
\]
and 
\[
	\frac{1}{n^2} \sum_{k = 1}^n \E[|Y_{n,k}|^4] \to 0 \text{ as }n \to \infty.
\]
The latter is trivial because $\E[|Y_{n,k}|^4]$ is uniformly bounded. 

Let us show the former. Recall that $\sigma_{n,k}^2$ depends on $\{a_{k \pm j}, b_{k \pm j} \}_{j = 0, \dots, m}$, where $m$ is the degree of $p$. Let $\sigma_{n,k}^2(\alpha)$ be the corresponding quantity of the i.i.d.~model $J_\alpha$. Since $b_{k \pm j} \sim \chitilde_{(n-k \mp j) \beta} = \chitilde_{(1 - \frac kn \mp \frac jn) 2\alpha}$, it is clear that 
\[
	\left|\E[\sigma_{n,k}^2] - \E[\sigma_{n, k}^2((1 - \frac kn)\alpha)]  \right| < c_{n,m} \to 0 \text{ as } n \to \infty, \text{for $m<k < n - m$}.
\]
Recall also that $\sigma_p^2((1 - \frac kn)\alpha) = \E[\sigma_{n, k}^2((1 - \frac kn)\alpha)]$, for $m<k < n - m$.
Consequently,
\[
	\frac{1}{n} \sum_{k = 1}^n \E[\sigma^2_{n,k}] \to \int_0^1 \sigma^2_p(u\alpha) du =: \sigma_p^2 \text{ as } n \to \infty.
\]
In addition, $\sigma_{n,i}^2$  and $\sigma_{n,j}^2$ are independent, if $|i - j| > 2m$. Thus 
\[
	\Var \left[\frac{1}{n} \sum_{k = 1}^n \sigma^2_{n,k} \right]= \frac{1}{n^2} \sum_{|i - j| \le 2m} \Cov(\sigma_{n,i}^2, \sigma_{n,j}^2)  = O(\frac 1n) \to 0 \text{ as } n \to \infty.
\]
Therefore 
\[
	\frac{1}{n} \sum_{k = 1}^n \sigma^2_{n,k}  \to \sigma_p^2 \text{ in probability as $n \to \infty$.} 
\]
Note that we have proved that 
\begin{equation}\label{convergence-of-variance}
n\Var[\bra{L_{n, \beta}, p}] = \frac{1}{n}\sum_{k = 1}^n \E[\sigma_{n,k}^2] \to \sigma_p^2 \text{ as } n \to \infty.
\end{equation}
We will need this property for the proof of Theorem~\ref{thm:CLT-general-f}.
The proof is complete. 
\end{proof}

To extend the CLT from polynomials to general test functions, we adopt the following approach which slightly improves an existence method quoted in \cite[Proposition~4.1]{Dumitriu-Paquette-2012}. First, base on the joint density of Gaussian beta ensembles, the following estimate is derived  (cf.~\cite[Eq.~(14)]{Duy2016}),
\begin{equation}\label{GE}
	n\Var[\langle L_{n, \beta}, f\rangle] \le \langle \bar L_{n, \beta}, (f')^2\rangle,
\end{equation}
for all continuous functions $f$ with continuous derivative. Next, in the regime that $n\beta = 2 \alpha$, as $n \to \infty$,
\[
	\lf \bar L_{n, \beta}, (f')^2\rt \to \lf \bar \mu_\alpha, (f')^2\rt < \infty,
\]
provided that $f \in C^1_p(\R)$, the set of differentiable function whose derivative $f'$ is a continuous function and $|f'(x)| \le P(x)$ for some polynomial $P$. Last, according to those estimates, the CLT for such function $f$ is derived by taking into account the following result.
\begin{lemma}[{\cite[Theorem 25.5]{Billingsley}}]\label{lem:triangle}
Let $\{Y_n\}_n$ and $\{X_{n,k}\}_{k,n}$ be real-valued random variables. Assume that
\begin{itemize}
	\item[\rm(i)] 
		$
			X_{n,k} \dto X_k \text{ as }n \to \infty;
		$
	\item[\rm(ii)]
		$
			X_k \dto X \text{ as } k \to \infty;
		$
	\item[\rm(iii)] for any $\varepsilon > 0$,
		$
			\lim_{k \to \infty} \limsup_{n \to \infty} \Prob(|X_{n,k} - Y_n| \ge \varepsilon) =0.
		$
\end{itemize}
Then $Y_n \dto X$ as $n \to \infty$.
\end{lemma}

We restate Theorem~\ref{CLT} here.
\begin{theorem}\label{thm:CLT-general-f}
Let $f \in C^1_p(\R)$. Then as $n \to \infty$ with $n \beta = 2 \alpha$,
\[
	\sqrt{n}(\lf L_{n, \beta}, f\rt - \E[\lf L_{n, \beta}, f\rt]) \dto \Normal(0, \sigma_f^2),
\]
for some constant 
$\sigma_f^2 \ge 0$.
\end{theorem}
\begin{proof}
Let $f \in C^1_p(\R)$.
Since the probability measure $\bar \mu_\alpha$ is determined by moments, the space of all polynomials is dense in $L^2(\bar \mu_\alpha)$ \cite[Corollary~2.50]{Deift-book-1999}. Take a sequence of polynomials $\{p_k\}_{k \ge 1}$ which converges to $f'$ in $L^2(\bar \mu_\alpha)$, and let $P_k$ be a primitive function of $p_k$, that is, $P_k' = p_k$. We claim that the limit 
\begin{equation}\label{limit-variance}
 	\sigma_f^2 = \lim_{n \to \infty} n \Var[\lf L_{n,\beta}, f\rt]
\end{equation}
exists and that 
\begin{equation}\label{variance-approximation}
	\sigma_{f}^2 = \lim_{k \to \infty} \sigma_{P_k}^2.
\end{equation}
Indeed, the estimate \eqref{GE} applying to $f - P_k$ yields,
\[	
	n\Var[\lf L_{n,\beta}, f - P_k \rt] \le  \lf \bar L_{n,\beta}, (f' - p_k)^2\rt .
\]
Letting $n \to \infty$, we have 
\[
	\limsup_{n \to \infty} n\Var[\lf L_{n,\beta}, f - P_k \rt] \le \lf \bar \mu_\alpha, (f' - p_k)^2\rt.
\]
Thus by the triangle inequality 
\begin{align*}
	&\limsup_{n \to \infty} (n \Var[\lf L_{n,\beta}, f\rt])^{1/2} \\
	&\le \lim_{n \to \infty} (n \Var[\lf L_{n,\beta}, P_k\rt])^{1/2} 
	 + \limsup_{n \to \infty} (n \Var[\lf L_{n,\beta}, f - P_k\rt])^{1/2}\\
	 &\le \sigma_{P_k} + \lf \bar \mu_\alpha, (f' - p_k)^2\rt^{1/2},
\end{align*}
and 
\begin{align*}
	&\liminf_{n \to \infty} (n \Var[\lf L_{n,\beta}, f\rt])^{1/2} \\
	 &\ge \lim_{n \to \infty} (n \Var[\lf L_{n,\beta}, P_k\rt])^{1/2} 
	 - \limsup_{n \to \infty} (n \Var[\lf L_{n,\beta}, f - P_k\rt])^{1/2}\\
	 &\ge \sigma_{P_k} - \lf \bar \mu_\alpha, (f' - p_k)^2\rt^{1/2}.
\end{align*}
Here we have used the equation~\eqref{convergence-of-variance} for the polynomial $P_k$.
It follows that the limit  
\[	
	\lim_{n \to \infty} (n \Var[\lf L_{n,\beta}, f\rt])^{1/2} =: \sigma_f
\]
exists, and that
\[
	|\sigma_f - \sigma_{P_k}| \le \lf\bar\mu_\alpha, (f' - p_k)^2\rt^{1/2} \to 0 \text{ as } k \to \infty,
\]
which proves the equations \eqref{limit-variance} and \eqref{variance-approximation}.

Let 
\begin{align*}
	Y_n &= {\sqrt{n}}(\langle L_{n,\beta}, f \rangle - \E[\langle L_{n,\beta}, f \rangle] ),\\
	X_{n,k} &= {\sqrt{n}} (\langle L_{n,\beta}, P_k \rangle - \E[\langle L_{n,\beta}, P_k \rangle] ).
\end{align*}
We are going to check three conditions in Lemma~\ref{lem:triangle}. Conditions (i) and (ii) are clear. Indeed, for any $k$, since $P_k$ is a polynomial, it follows from Theorem~\ref{thm:poly} that
\[
	X_{n, k} \dto \Normal(0, \sigma_{P_k}^2) =: X_k \text{ as } n \to \infty.
\]
In addition, $X_k$ converges in distribution to $\Normal (0, \sigma_f^2)$ as $k \to \infty$ by the equation~\eqref{variance-approximation}.

 For the 
condition (iii), recall that 
\[
	\Var[X_{n,k} - Y_n] \le \lf \bar L_{n,\beta}, (f' - p_k)^2\rt,
\]
and hence 
\[
	\lim_{n \to \infty} \Var[X_{n,k} - Y_n] \le \lf \bar \mu_\alpha, (f' - p_k)^2\rt \to 0 \as k \to \infty.
\]
Therefore, for any $\varepsilon > 0$,
		\begin{align*}
			\lim_{k \to \infty} \limsup_{n \to \infty} \Prob(|X_{n,k} - Y_n| \ge \varepsilon) \le \lim_{k \to \infty} \limsup_{n \to \infty} \frac{1}{\varepsilon^2} \Var[X_{n,k} - Y_n] = 0.
		\end{align*}
The theorem is proved. 
\end{proof}

\subsubsection{Local law}
The purpose of this subsection is to establish Condition L$'$ with $\theta = \bar \mu_\alpha (E)$ for Gaussian beta ensembles in the regime that $n \beta = 2 \alpha$. Note that Condition G holds as a consequence of the general result in Section~\ref{sec:G}. Therefore once Condition L$'$ is proved, Theorem~\ref{Poisson} follows, namely, we have
\begin{theorem}
As $n \to \infty$ with $n \beta = 2\alpha$, the point process 
\[
	\xi_n = \sum_{j = 1}^n \delta_{n(\lambda_j - E)}
\] 
converges weakly to a homogeneous Poisson point process with intensity $\bar \mu_\alpha(E)$.
\end{theorem}

To show the local law, we need some preliminary notations and results. The $m$-function of a Jacobi matrix $J$ (finite or infinite) is defined as $m(z) = (J - z)^{-1}(1,1)$, where $J$ is required to be essentially self-adjoint in the infinite case. Note that the $m$-function is nothing but the Stieltjes transform of the spectral measure of $J$. Denote by $m_n(z), m_\alpha(z)$ and $m_{\alpha, n}(z)$ the $m$-functions of $T_{n,\beta}, J_\alpha$ and $J_{\alpha}^{[1,n]}$, respectively. Since the spectral measure and the empirical distribution of Gaussian beta ensembles have the same mean, we have
\begin{equation}
	\E[ \xi_n( f_\zeta)] = \E[\Image m_n(E + \frac{\zeta}{n})]. 
\end{equation}
Let $G_n(z), G_\alpha(z)$ and $G_{\alpha, n}(z)$ be the Green's functions of $T_{n,\beta}, J_\alpha$ and $J_{\alpha,n}$, respectively. We recall a result on the exponential decay of Green's functions from Theorem~\ref{thm:exponential-decay}.

\begin{lemma}\label{lem:bounded-estimate}
Let $G(z)$ stand for any of $G_n(z), G_\alpha(z)$ and $G_{\alpha, n}(z)$. Then for $0<s<1/4$, there are positive constants $M$ and $\gamma$ such that 
	\[
			\E[|G(\lambda + i \tau; 1, x)|^s] \le M e^{-\gamma (x - 1)},
	\]	
		for $\lambda \in [-\Lambda, \Lambda]$ and $\tau > 0$.
\end{lemma}

\begin{lemma}\label{lem:convergence-of-malpha}
Let $\delta > 0$ and $\zeta = \sigma + i \tau \in \C_+$ be given. Then the following statements hold.
\begin{itemize}
	\item[\rm (i)] As $n \to \infty$, 
	\[
		\E[|m_{\alpha, n}(E + \frac{\zeta}{n^\delta}) - m_\alpha(E + \frac{\zeta}{n^\delta})|] \to 0. 
	\]
	\item[\rm (ii)] $\E[\Image m_\alpha(E + \frac{\zeta}{n^\delta})] = \Image \E[m_\alpha(E + \frac{\zeta}{n^\delta})] \to \pi \bar \mu_\alpha(E)$ as $n \to \infty$.
	\item[\rm(iii)] Consequently, $\E[\Image m_{\alpha, n}(E + \frac{\zeta}{n^\delta})] \to \pi \bar \mu_\alpha(E)$ as $n \to \infty$.
\end{itemize}

\end{lemma}
\begin{proof}
(i) easily follows from the resolvent equation and the exponential decay of Green's functions. (ii) follows from the fact that $\E[m_\alpha(z)]$ is the Stieltjes transform of $\bar \mu_\alpha$, a probability measure with continuous density. (iii) is a direct consequence of (i) and (ii).
\end{proof}

\begin{lemma}\label{lem:delta}
Let $\zeta = \sigma + i \tau \in \C_+$ be given. Then for $0 <\delta \le 1/4$, 
\[
	\E[|m_n(E + \frac{\zeta}{n^\delta}) - m_{\alpha, n}(E + \frac{\zeta}{n^\delta})|] \to 0 \text{ as $n \to \infty$ with $n\beta = 2 \alpha$},
\]
and hence
\begin{equation}\label{local-law-at-delta}
		 \E[ \Image m_n(E + \frac{\zeta}{n^\delta})] \to \pi \bar\mu_\alpha (E).
	\end{equation}
\end{lemma}

\begin{proof}
We use the following coupling: 
\begin{align*}
\{a_j\}_{j = 1}^n &\sim \Normal(0,1); \\
\{b_j^2\}_{j = 1}^{n - 1} &\sim \Gam \left(\frac{(n - j) \beta}{2}, 1 \right) = \Gam \left(\alpha(1 - \frac{j}n), 1\right);\\
\{c_j^2\}_{j = 1}^{n - 1} &\sim \Gam \left(\frac{ \alpha j}{n}, 1\right);\\
d_j^2 &:= b_j^2 + c_j^2 \sim \Gam (\alpha, 1);
\end{align*}
	\[
	T_{n,\beta} =  \begin{pmatrix}
		a_1		&b_1	\\
		b_1		&a_2		&b_2		\\
					
												&\ddots		&\ddots		&\ddots \\
		&& b_{n - 1} &a_n
			\end{pmatrix};\quad 
			J_{\alpha, n} =  \begin{pmatrix}
		a_1		&d_1	\\
		d_1		&a_2		&d_2		\\
					
												&\ddots		&\ddots		&\ddots \\
		&& d_{n - 1} &	a_n
			\end{pmatrix}.
\]
Then by the resolvent equation, 
\begin{align*}
	&m_n(z) - m_{\alpha, n}(z) = G_n(z;1,1) - G_{\alpha, n}(z; 1,1) \\
	&=\sum_{x = 1}^{n - 1} (d_x - b_x)\Big\{ G_n(z; 1, x)G_{n, \alpha}(z; x+1, 1) + G_n(z; 1, x + 1)G_{n, \alpha}(z; x, 1) \Big\}.
\end{align*}
A general term in the above sum can be estimated as follows 
\begin{align*}
	&\E[|(d_x - b_x)G_n(z; 1, x)G_{n, \alpha}(z; x+1, 1)|] \\
	&\le \frac{1}{(\Image z)^{2 -  s/2}} \E[c_x | G_{n, \alpha}(z; x+1, 1)|^{s/2}] \\
	&\le \frac{1}{(\Image z)^{2 - s/2}} \E[c_x^2]^{1/2} \E[| G_{n, \alpha}(z; x+1, 1)|^{s}]^{1/2}  \\
	&\le \frac { n^{\delta (2 - s/2)}}{\tau^{2 - s/2}} \left(\frac{\alpha x}{n}\right)^{1/2} M^{1/2} e^{-\gamma x /2} \\
	&= \tilde M  n^{-\varepsilon} \sqrt x e^{-\gamma x /2}.
\end{align*}
Here $z = E + \zeta/n^{\delta}$, $\varepsilon = 1/2 - 2\delta + \delta s/2 > 0$, $M$ and $\gamma$ are the constants in Lemma~\ref{lem:bounded-estimate}. Therefore 
\[
	\E[|m_n(z) - m_{\alpha, n}(z)|] \le 2 \tilde M  n^{-\varepsilon} \sum_{x = 1}^{\infty} \sqrt x e^{-\gamma x /2} \to 0 \text{ as } n\to \infty.
\]
The proof is complete.
\end{proof}

Our next aim is to extend the equation~\eqref{local-law-at-delta} to hold for any $\delta > 0$. Note that the local law is nothing but the equation~\eqref{local-law-at-delta} with $\delta = 1$. Our argument is based on the following result.

\begin{lemma}\label{lem:Stieltjes-transform}
Assume that $\{\mu_n\} $ is a sequence of  probability measures whose densities $\mu_n(x)$ are uniformly bounded, that is, $\mu_n(x) \le C$, for all $x \in \R$ and all $n$. Let $E$ be fixed and $\{\tau_n\}$ be a sequence of positive numbers tending to infinity. Assume that for any $M > 0$, 
	\begin{equation}\label{uniform-continuity}
		\sup_{|t| \le \frac{M}{\tau_n}} |\mu_n(E + t) - \mu_n(E)| \to 0 \text{ as } n \to \infty.
	\end{equation}
Then for any $\zeta = \sigma + i \tau \in \C_+$,
\[
	\Image S_{\mu_n}(E + \frac{\zeta}{\tau_n}) - \pi \mu_n(E) \to 0   \text{ as } n \to \infty.
\]
Here recall that $S_{\mu_n}(z)$ denotes the Stieltjes transform of $\mu_n$.
\end{lemma}
\begin{proof}
It follows from the definition of the Stieltjes transform that 
	\begin{align*}
		\Image S_{\mu_n}(E + \frac{\zeta}{\tau_n}) - \pi \mu_n(E) &=  \int_\R \frac{t_n \mu_n(x) dx}{(x - E - \sigma_n)^2 + t_n^2} - \pi \mu_n(E) \\
		&= \int_\R \frac{\mu_n(E + \sigma_n + t_n y)dy}{1 + y^2} - \int_\R \frac{\mu_n(E)dy}{1 + y^2}\\
		&= \int_\R \frac{(\mu_n(E + \sigma_n + t_n y) - \mu_n(E))dy}{1 + y^2},
	\end{align*}
where $\sigma_n = \frac{\sigma}{\tau_n}$ and $t_n = \frac{\tau}{\tau_n}$. Given $\varepsilon > 0$, we first choose an $M > 0$ such that 
\[
	\int_{|y| > M} \frac{1}{1 + y^2} < \varepsilon,
\]
and then choose an $n_\varepsilon$ such that for $n > n_\varepsilon$,
\[
	\sup_{|t| \le \frac{|\sigma| + \tau M}{\tau_n}} |\mu_n(E + t) - \mu_n(E)| < \varepsilon .
\]
Now for $n > n_\varepsilon$, it is clear that 
\begin{align*}
		|\Image S_{\mu_n}(E + \frac{\zeta}{\tau_n}) - \pi \mu_n(E)| &\le  \int_\R \frac{|\mu_n(E + \sigma_n + t_n y) - \mu_n(E)|dy}{1 + y^2}\\
		&= \int_{|y| \le M}(\cdots) + \int_{|y| > M}(\cdots) \\
		&\le \int_{|y| \le M} \frac{\varepsilon dy}{1 + y^2}  + \int_{|y|>M}\frac{2Cdy}{1 + y^2}\\
		&\le (\pi + 2C)\varepsilon.
	\end{align*}
Therefore $\Image S_{\mu_n}(E + \frac{\zeta}{\tau_n}) - \pi \mu_n(E) \to 0$, which completes the proof. 
\end{proof}

We now show that for any $\delta > 0$, the condition \eqref{uniform-continuity} holds for  $\{\tau_n = n^\delta\}$ with respect to the sequence of the mean measures $\{\bar \nu_{n, \beta}\}$ in the regime that $n\beta \to 2\alpha$. Recall that the mean measure $\bar \nu_{n, \beta}$ coincides with the mean of the empirical measure $L_{n,\beta}$. Thus, its density $\bar \nu_{n, \beta}(E)$ is given by 
\[
	\bar \nu_{n, \beta}(E) = \int_{\R^{n - 1}} p_{n, \beta} (\lambda_1,  \dots, \lambda_{n-1}, E) d\lambda_1 \cdots d\lambda_{n-1} ,
\]
where $p_{n, \beta}(\lambda)$ is the joint density of Gaussian beta ensembles
\[
	p_{n, \beta}(\lambda) = p_{n, \beta}(\lambda_1, \lambda_2, \dots, \lambda_n) = \frac{1}{Z_{n, \beta}} |\Delta(\lambda)|^\beta e^{-\frac12(\lambda_1^2 + \cdots + \lambda_n^2)},
\]
with 
\[
	Z_{n, \beta} = (2\pi)^{\frac{n}{2}} \prod_{j = 1}^n \frac{\Gamma(1 + j\beta/2)}{\Gamma(1 + \beta/2)}.
\]
We further express the density $\bar \nu_{n, \beta}(E)$ as follows
\begin{align*}
	\bar \nu_{n,\beta}(E) 
	&= \frac{Z_{n-1, \beta}}{Z_{n, \beta}} e^{-\frac12 E^2} \int_{\R^{n - 1}} \prod_{1\le j \le n-1} |E - \lambda_j|^\beta p_{n-1,\beta}(\lambda) d\lambda.
\end{align*}
Note that Wegner's estimate implies that $\bar \nu_{n,\beta}(E) \le \pi M_A =\sqrt{\pi/2}$. However, we will derive an upper bound for the density $\bar \nu_{n,\beta}(E)$ directly from the above expression.

Assume that $n \beta \le \kappa$, where $\kappa\in \{2,4,6,\dots\}$. By using the following inequality
\[
	x_1^{\alpha_1} \cdots x_n^{\alpha_n} \le \alpha_1 x_1 + \cdots + \alpha_n x_n,  \quad (x_i \ge 0, \alpha_i > 0, \alpha_i + \cdots + \alpha_n = 1)
\]
with $\alpha_i = \beta/\kappa < \frac{1}{n-1}, i = 1, \dots, n-1$ and $\alpha_n = (1 - (n-1)\beta/\kappa)$, we obtain that for $n \ge 2$,
\begin{align*}
	 \prod_{1\le j \le n-1} |E - \lambda_j|^\beta  \le \alpha_n + \sum_{j = 1}^{n - 1} \alpha_j |E-\lambda_j|^{\kappa} &\le 1 + \frac{1}{n - 1} \sum_{j = 1}^{n - 1}2^{\kappa - 1} (E^\kappa + \lambda_j^\kappa)\\
	  &\le 2^{\kappa - 1}  \left(1 + E^{\kappa} + \frac{1}{n - 1}\sum_{j = 1}^{n-1}\lambda_j^{\kappa} \right).
\end{align*}
Thus,
\begin{equation}\label{bound-for-density}
	\bar \nu_{n,\beta}(E)  \le \frac{Z_{n-1, \beta}}{Z_{n, \beta}} e^{-\frac12 E^2} 2^{\kappa - 1} (1 + E^{\kappa} + \Ex[\bra{L_{n-1,\beta}, x^{\kappa} }]).
\end{equation}
It is clear that under the condition $n\beta \le \kappa$, $\Ex[\bra{L_{n-1,\beta}, x^{\kappa} }]$ is uniformly bounded. In addition, as $n \to \infty$ with $n\beta \to 2\alpha$,
\[
	 \frac{Z_{n-1, \beta}}{Z_{n, \beta}} = \frac{1}{\sqrt{2\pi}}\frac{\Gamma(1 + \beta/2)}{\Gamma(1 + n\beta/2)} \to \frac{1}{\sqrt{2\pi} \Gamma(1 + \alpha)} .
\]
Thus the density $\bar\nu_{n,\beta}(E)$ is uniformly bounded, that is, $\bar \nu_{n, \beta}(E) \le C = C(\kappa)$.

 We now study the difference $\bar \nu_{n,\beta} (E + t) - \bar \nu_{n, \beta} (E)$ under the condition that $n\beta \le \kappa$ with $\beta < 1$, and $|t| \le 1/2$. To begin with, we estimate roughly as follows
\begin{align}
	&|\bar \nu_{n,\beta}(E + t) -\bar \nu_{n,\beta}(E)| \nonumber\\
	&\le \frac{Z_{n-1, \beta}}{Z_{n, \beta}} |e^{-\frac12 (E+t)^2} - e^{-\frac12 E^2}| \int_{\R^{n - 1}} \prod_{1\le j \le n-1} |E +t - \lambda_j|^\beta p_{n-1,\beta}(\lambda) d\lambda\nonumber\\
	&\quad + \frac{Z_{n-1, \beta}}{Z_{n, \beta}} e^{-\frac12 E^2} \int_{\R^{n - 1}} \bigg| \prod_{1\le j \le n-1} |E +t - \lambda_j|^\beta - \prod_{1\le j \le n-1} |E - \lambda_j|^\beta \bigg| p_{n-1,\beta}(\lambda) d\lambda \nonumber\\
	&\le C\bigg( |t| +   \int_{\R^{n - 1}} \bigg| \prod_{1\le j \le n-1} |E +t - \lambda_j|^\beta - \prod_{1\le j \le n-1} |E - \lambda_j|^\beta \bigg| p_{n-1,\beta}(\lambda) d\lambda \bigg). \label{difference-estimate}
\end{align}
Here $C=C(E, \kappa)$ is a constant.

\begin{lemma}\label{lem:key-estimate}
There is a constant $C = C(E, \kappa)$ such that for $n\beta \le \kappa$  with $\beta < 1$, and $|t|<1/2$, 
	\begin{align*}
		E_{n,i}(t) &:= \int_{\R^{n - 1}} \big| |E+t-\lambda_i|^\beta - |E -\lambda_i|^\beta \big| \prod_{j < i} |E +t - \lambda_j|^\beta \prod_{j > i} |E - \lambda_j|^\beta  p_{n-1,\beta}(\lambda) d\lambda \\
		&\le C(\beta |t| + |t| (1 - |\frac t 2|^\beta)).
	\end{align*}
\end{lemma}
\begin{proof}
We use the same argument as in proving the estimate \eqref{bound-for-density} to deduce that
\[
\frac{Z_{n-2,\beta}}{Z_{n-1, \beta}}\int_{\R^{n - 2}} \prod_{j < i} |E +t - \lambda_j|^\beta \prod_{j > i} |E - \lambda_j|^\beta \prod_{j \ne i} |\lambda_i - \lambda_j|^\beta  p_{n-2,\beta}(\lambda) d\lambda \le  M_1(1 + \lambda_i^{2\kappa}),
\]
where $M_1 = M_1(E, \kappa)$ is a constant. Thus,
\begin{align*}
	E_{n,i}(t) 
	&\le M_1 \int_\R \big| |E+t-\lambda_i|^\beta - |E -\lambda_i|^\beta \big|  (1 + \lambda_i^{2\kappa}) e^{-\frac12 \lambda_i^2} d\lambda_i \\
	&=M_1 \left( \int_{|E - \lambda_i| \le 1}(\cdots) + \int_{|E - \lambda_i| > 1}(\cdots) \right) =: M_1 ((I) + (II)).
\end{align*}
In case $|E - \lambda_i| > 1$, it follows from the mean value theorem that  
\[
	\big| |E+t-\lambda_i|^\beta - |E -\lambda_i|^\beta \big| \le 2\beta |t|,
\]
and hence 
\[
	(II) \le \int_{|E - \lambda_i| > 1} 2\beta |t|  (1 + \lambda_i^{2\kappa}) e^{-\frac12 \lambda_i^2} d\lambda_i \le 2\beta |t| \int_\R (1 + \lambda_i^{2\kappa}) e^{-\frac12 \lambda_i^2} d\lambda_i =: M_2 \beta |t|.
\]

When $|E - \lambda_i| \le 1$, it is clear that 
\[
	(I) \le M_3 \int_{|E - \lambda_i| \le 1}  \big| |E+t-\lambda_i|^\beta - |E -\lambda_i|^\beta \big| d\lambda_i = M_3 \int_{-1}^1  \big| |x + t|^\beta - |x|^\beta \big| dx,
\]
where $M_3 = \sup_{\lambda_i} (1 + |\lambda_i|^{2\kappa}) e^{-\frac12 \lambda_i^2}$.
The last integral is easily calculated and its value is given by 
\[
	\int_{-1}^1  \big| |x + t|^\beta - |x|^\beta \big| dx = \frac{1}{\beta + 1} ((1 + |t|)^{1 + \beta} - (1 - |t|)^{1 + \beta} - 2^{1 -\beta}|t|^{\beta + 1}).
\]
For $0 < \beta < 1$ and $|t| \le 1/2$, it follows from Taylor's theorem that
\begin{align*}
	&(1 + |t|)^{1 + \beta} \le 1 + (1 + \beta)|t| + \frac{t^2}{2}\beta(1 + \beta),\\
	&(1 - |t|)^{1 + \beta} \ge 1 - (1+ \beta) |t|.
\end{align*} 
The  lemma follows by collecting all the above estimates. 
\end{proof}

Since the second term in the estimate \eqref{difference-estimate} is bounded by the sum of $\{E_{n,i}(t)\}_{i = 1}^{n-1}$, Lemma~\ref{lem:key-estimate} implies the following 
\begin{lemma}\label{lem:final-estimate}
There is a constant $C = C(E, \kappa)$ such that for $n\beta \le \kappa$ with $\beta < 1$ and  $|t| < 1/2$, 
	\begin{equation}\label{final-estimate}
	|\bar \nu_{n,\beta}(E + t) -\bar \nu_{n,\beta}(E)| \le C(|t| +  n|t| (1 - |\frac t 2|^\beta)).
	\end{equation}
\end{lemma}

\begin{lemma}\label{lem:local-continuity}
Let $\delta_n$ be a sequence of positive numbers such that as $n \to \infty$,
\[
 \delta_n \to 0;  \delta_n \log \delta_n \to 0.
\]
Then in the regime that $n\beta \to 2\alpha$, 
\[
	\sup_{|t| \le \delta_n} |\bar \nu_{n, \beta}(E + t) - \bar \nu_{n, \beta}(E)| \to 0.
\]
In particular, for any $\delta > 0$ and  any $M > 0$, as $n \to \infty$ with $n\beta \to 2\alpha$,
	\[
		\sup_{|t| \le \frac{M}{n^\delta}} |\bar \nu_{n, \beta}(E + t) - \bar \nu_{n, \beta}(E)| \to 0.
	\]
\end{lemma}
\begin{proof}
	For $n$ large enough such that $\delta_n < 1/2$ and $\beta < 1$, it follows from Lemma~\ref{lem:final-estimate} that
	\[
		\sup_{|t| \le {\delta_n}} |\bar \nu_{n, \beta}(E + t) - \bar \nu_{n, \beta}(E)| \le \sup_{|t| \le {\delta_n}}   C(|t| +  n|t| (1 - |t/ 2|^\beta)) = C(\delta_n +  n \delta_n (1 - (\delta_n/2)^\beta ) ).
	\]
Then the desired result follows from the assumption with the help of the following inequality 
\[
	1 - (\delta_n/2)^\beta = 1 - \exp(\beta (\log \delta_n -\log 2)) \le -\beta ( \log \delta_n - \log 2).
\]
The proof is complete. 
\end{proof}

\begin{theorem}[Local law]\label{thm:local-law}
As $n \to \infty$ with $n\beta \to 2 \alpha$,
		\[
			\E[ \xi_n( f_\zeta)] = \E[\Image m_n(E + \frac{\zeta}{n})] \to \pi \bar \mu_\alpha(E).
		\]
\end{theorem}

\begin{proof}
It follows from Lemma~\ref{lem:local-continuity} that for any $\delta > 0$, the sequence $\tau_n = n^\delta$ satisfies the assumption in Lemma~\ref{lem:Stieltjes-transform} with respect to the sequence of probability measures $\bar \nu_{n, \beta}$. Thus for any $E \in \R$ and any $\zeta \in \C_+$, as $n \to \infty$ with $n\beta \to 2\alpha$,
\[
	\E[\Image m_{n}(E + \frac{\zeta}{n^\delta})] - \pi \bar \nu_{n, \beta}(E) = \Image S_{\bar \nu_{n,\beta}}(E + \frac{\zeta}{n^\delta}) - \pi \bar \nu_{n, \beta}(E) \to 0.
\]
On the other hand, for $0 < \delta \le 1/4$, Lemma~\ref{lem:delta} claims that $\Ex[\Image m_n(E + \frac{\zeta}{n^\delta})] \to \pi \bar \mu_\alpha (E)$. Thus 
\[
	\bar \nu_{n, \beta}(E) \to \bar \mu_\alpha (E) \text{ as } n \to \infty.
\]
Consequently, for any $\delta > 0$,
\[
	\Image m_n(E + \frac{\zeta}{n^\delta}) \to \bar \mu_\alpha (E) \text{ as } n \to \infty.
\]
The local law is just the case $\delta = 1$. The proof is complete.
\end{proof}

The rest of this subsection is devoted to prove the identity 
\[
	\bar \mu_\alpha(E) = \frac{1}{\sqrt{2 \pi} \Gamma(\alpha + 1)} \exp\left(-\frac{E^2}{2} + 2 \alpha \int \log |E - y| \bar\mu_\alpha(dy)  \right),
\]
which is derived by showing the existence of the limit of $\{\bar \nu_{n, \beta}(E)\}$ in another way.
\begin{lemma}[{cf.~\cite{Peche15}}]
As $n \to \infty $ with $n\beta \to 2 \alpha$, 
\[
	\bar \nu_{n, \beta}(E) \to  \frac{1}{\sqrt{2 \pi} \Gamma(\alpha + 1)} \exp\left(-\frac{E^2}{2} + 2 \alpha \int \log |E - y| \bar\mu_\alpha(dy)  \right).
\]
\end{lemma}

\begin{proof}
Recall that 
\[
	\bar \nu_{n,\beta}(E) = \frac{Z_{n-1, \beta}}{Z_{n, \beta}} e^{-\frac12 E^2} \int_{\R^{n - 1}} \prod_{1\le j \le n-1} |E - \lambda_j|^\beta p_{n-1,\beta}(\lambda) d\lambda,
\]
which can be rewritten as 
\[
	\bar \nu_{n,\beta}(E) =  \frac{1}{\sqrt{2\pi}}\frac{\Gamma(1 + \beta/2)}{\Gamma(1 + n\beta/2)}  e^{-\frac12 E^2} \E\left[\exp\left(\beta (n-1)\bra{L_{n-1,\beta}, {\log|E - \cdot|}}  \right) \right].
\]
Although the function $\log|E - x|$ is not continuous, we will show in Lemma~\ref{lem:log} that 
\[
	\bra{L_{n-1,\beta}, {\log|E - \cdot|}} \to \bra{\bar \mu_\alpha, \log|E - \cdot|} = \int_{\R} \log |E - y| \bar \mu_\alpha(y) dy \text{ in probability}
\]
by a truncation argument. It then follows by the continuous mapping theorem that 
	\[
		\prod_{1\le j \le n-1} |E - \lambda_j|^\beta = \exp({\beta (n-1) \bra{L_{n-1,\beta}, {\log|E - \cdot|}} }) 
		\to \bra{\bar \mu_\alpha, \log|E - \cdot|}  \text{ in probability.}
	\]
In addition, by the same argument as in proving the equation \eqref{bound-for-density}, we can also show that  for $n\beta \le \kappa$,
\[
	\Ex_{n -1, \beta}\bigg[\bigg (\prod_{1\le j \le n-1} |E - \lambda_j|^\beta \bigg)^2\bigg] \le C = C(E, \kappa).
\]
This implies that the sequence $\{\prod_{1\le j \le n-1} |E - \lambda_j|^\beta\}$ is uniformly integrable, and hence, the convergence of the expectation follows. Therefore, as $n \to \infty$ with $n\beta \to 2\alpha$,
\begin{equation*}\label{convergence-of-density}
	\bar\nu_{n,\beta}(E) \to  \frac{1}{\sqrt{2 \pi} \Gamma(\alpha + 1)} \exp\left(-\frac{E^2}{2} + 2 \alpha \int \log |E - y| \bar\mu_\alpha(dy)  \right),
\end{equation*}
which completes the proof. 
\end{proof}

\begin{lemma}\label{lem:log}
As $n \to \infty$ with  $n\beta \to 2 \alpha$, 
	\[
		\bra{L_{n, \beta}, \log|E-\cdot|} =\frac{1}{n} \sum_{j = 1}^n  \log |E - \lambda_j| \to \int_{\R} \log |E - x| \bar \mu_\alpha(x) dx \text{ in probability}.
	\]
\end{lemma}
\begin{proof}
	For $M > 0$, let 
	\[
		f_M(x) = \begin{cases}
			\log |E - x|, &\text{if } |E - x| \ge e^{-M},\\
			-M, &\text{if } |E - x| \le e^{-M}.
		\end{cases}
	\]
Then $f_M$ is a continuous function on $\R$ of polynomial growth. Thus, in the regime that $n\beta \to 2 \alpha$,
\[
	\frac{1}{n} \sum_{j = 1}^n f_M(\lambda_j) \to \int_\R f_M(x) \bar \mu_\alpha(x) dx \text{ in probability}.
\]
Note that the convergence also holds in $L^1$.

On the other hand, recall that $\bar\nu_{n, \beta}(x) \le C$ for all $x \in \R$. Thus 
\begin{align*}
	0 \le \Ex[\bra{L_{n, \beta}, f_M} - \bra{L_{n, \beta}, \log|E - \cdot|}] &= \int_{\R} (f_M(x) - \log|E - x|) \bar\nu_{n, \beta}(x) dx \\
	&\le C \int_{|E - x| \le e^{-M}} (-\log |E - x|  -M)dx \\
	&= 2C e^{-M}.
\end{align*}
Since the density $\bar \mu_\alpha(x)$ is also bounded by $C$, by the same estimate, we have 
\[
	0 \le \int_{\R} (f_M(x) - \log|E - x|)\bar \mu_\alpha (x) dx \le 2C e^{-M}.
\]
Thus by a standard argument using the triangular inequality, we deduce that 
\[
	\frac{1}{n} \sum_{j = 1}^n  \log |E - \lambda_j| \to \int_{\R} \log |E - x| \bar \mu_\alpha(x) dx \text{ in $L^1$, and hence in probability.}
\]
The proof is complete.
\end{proof}

\appendix
\section{Martingale difference central limit theorem}\label{app:CLT}
Suppose that, for each $n$, $X_{n1}, X_{n2},\dots$ is a martingale with respect to $\F_{n1}, \F_{n2}, \dots$. Define $Y_{nk} = X_{nk} - X_{n, k - 1}$. Suppose the $Y_{nk}$ have second moments, and put $\sigma_{nk}^2 = \E[Y_{nk}^2 | \F_{n, k -1}]$ ($\F_{n0} = \{\emptyset, \Omega\}$). The probability space may vary with $n$. If the martingale is originally defined only for $1 \le k \le r_n$, take $Y_{nk} = 0$ and $\F_{nk} = \F_{n r_n}$ for $k > r_n$. Assume that $\sum_{k =1}^\infty Y_{nk}$ and $\sum_{k = 1}^\infty \sigma_{nk}^2$ converge with probability $1$.
\begin{theorem}[{\cite[Theorem~35.12]{Billingsley}}]
Assume that 
\[
	\sum_{k = 1}^\infty \sigma_{nk}^2 \to \sigma^2 \text{ in probability as }n \to \infty,
\]
where $\sigma^2$ is a positive constant, and that
\[
	\sum_{k = 1}^\infty \E[Y_{nk}^2 ;{|Y_{nk}| \ge \varepsilon}] \to 0
\]
for each $\varepsilon >0$. Then $\sum_{k = 1}^\infty Y_{nk} \overset{d}{\to} \mathcal N(0, \sigma^2)$.
 
\end{theorem}

In this paper, we use the following version which is an easy consequence of the previous theorem. Let $\{S_n\}_{n = 1}^\infty$ be a sequence of random variables. For each $n$, we consider some filtrations $\{\emptyset, \Omega\} = \F_{n0} \subset \F_{n1} \subset \cdots \subset \F_{nn}$. Let $X_{nk} = \E[S_n | \F_{nk}], (0 \le k \le n)$. Define $Y_{nk} = X_{nk} - X_{n, k-1}$ and $\sigma_{nk}^2 = \E[Y_{nk}^2 | \F_{n, k-1}]$ for $1\le k \le n$. 
\begin{theorem}\label{thm:MTG}
	Assume that the following two conditions holds 
\begin{itemize}
\item[\rm (i)]
	$
		\frac{1}{n} \sum_{k = 1}^n \sigma_{nk}^2 \to \sigma^2 \text{ in probability as $n \to \infty$},
	$
	where $\sigma^2 \ge 0$ is a contant; 
\item[\rm (ii)]
	$
		\frac{1}{n^2} \sum_{k = 1}^n \E[|Y_{n,k}|^4] \to 0 \text{ as } n \to \infty.
	$
\end{itemize}
	Then 
	\[
		\frac{S_n - \E[S_n]}{\sqrt{n}} \dto \Normal(0, \sigma^2).
	\]
\end{theorem}

\section{Convergence of random probability measures on the real line}\label{app:rpm}
Let $\cP(\R)$ be the space of all probability measures on $(\R, \cB(\R))$, where $\cB(\R)$ denotes the Borel $\sigma$-field of $\R$. A sequence of probability measures $\{\mu_n \}_{n = 1}^\infty$ is said to converge weakly to $\mu \in \cP(\R)$ if for all bounded continuous functions $f \colon \R \to \R$ (or $\C$), 
\[
	\lim_{n \to \infty} \int_\R f d\mu_n = \int_\R f d\mu.
\]
The topology of weak convergence on $\cP(\R)$ can be metrizable by the L\'evy--Prokhorov metric $\rho$, which makes $(\cP(\R), \rho)$  a separable complete metric space. We do not need the precise definition of the metric here. Let $\cB(\cP(S))$ be the Borel $\sigma$-field on $\cP(S)$.  

\begin{definition}
A random probability measure $\xi$ is a measurable map from some probability space $(\Omega, \F, \Prob)$ to $(\cP(S), \cB(\cP(S))$.
\end{definition}

Let $\xi$ be a random probability measure. Then for any Borel set $B \in \cB(\R)$, $\xi(B)$ is a usual random variable. So is $\lf \xi, f\rt$ for any non negative measurable function $f$, or bounded measurable function $f$, where $\bra{\xi, f} = \int f d\xi$.  As random variables on a metric space, concepts of almost sure convergence, convergence in probability and convergence in distribution of random probability measures are defined naturally. 
\begin{definition}
\begin{itemize}
\item[(i)]
	Let $\{\xi_n\}_{n = 1}^\infty$ and $\xi$ be random probability measures defined on the same probability space. The sequence $\{\xi_n\}$ is said to converge weakly to $\xi$ almost surely (resp.~in probability) if $\rho(\xi_n, \xi)$ converges almost surely (resp.~in probability) to $0$ as $n$ tends to infinity.
\item[(ii)] 
	Let $\{\xi_n\}_{n =1}^\infty$ and $\xi$ be random probability measures which may be defined on different probability spaces. The sequence $\{\xi_n\}$ is said to converge in distribution to $\xi$ if for any bounded continuous function $\Phi \colon \cP(\R) \to \R$ (or $\C$), 
	\[
		\E[\Phi(\xi_n)] \to \E[\Phi(\xi)] \text{ as } n \to \infty.
	\]
	
\end{itemize}
\end{definition}
The aim of this section is to show that three abstract concepts of convergence can be defined in term of usual concepts of convergence of $\lf \xi_n, f \rt$ for all bounded continuous function $f$. We need some preliminaries.

A class $\cA$ of measurable functions on $\R$ is called a convergence determining class if for any sequence of probability measures $\{\mu_n\}_{n = 1}^\infty$ and any probability measure $\mu$, the condition 
\[
	\lim_{n \to \infty} \int_\R f d\mu_n = \int_\R f d\mu, \text{ for all $f \in \cA$},
\]
implies that $\{\mu_n\}$ converges weakly to $\mu$. By definition, the space $C_b(\R)$ of bounded continuous function on $\R$ is an example of convergence determining class.

Let $\mu$ be a probability measure. A class $\cA_\mu$ of functions on $\R$ is called a convergence determining class for $\mu$ if for any sequence of probability measures $\{\mu_n\}_{n = 1}^\infty$,  the condition 
\[
	\lim_{n \to \infty} \int_\R f d\mu_n = \int_\R f d\mu, \text{ for all $f \in \cA_\mu$},
\]
implies that $\{\mu_n\}$ converges weakly to $\mu$.

Let $S_\mu (z)$ denote the Stieltjes transform of $\mu$, 
\[
	S_\mu (z) = \int_\R \frac{1}{x - z} \mu(dx), z \in \C_+.
\]
We have the following estimates:
	\begin{align*}
		|S_\mu(z)| &\le \frac{1}{\Image z};\\
		|S_\mu(z) - S_\mu(z')| &\le \frac{|z - z'|}{\Image z \Image z'}.
	\end{align*}
For $z \in \C_+$, denote by $f_z = 1/(x - z)$. Note that $\mu_n$ converges weakly to $\mu$, if and only if $S_{\mu_n}(z)$ converges to $S_\mu(z)$ for all $z \in \C_+$, which means that $\{f_z\}_{z \in \C_+}$ is a convergence determining class.
Let $D$ be a countable dense subset in $\C_+$. Then using the above estimates, we can show that the class 
\[
	\cA = \{f_{z}\}_{z \in D}
\]
is a countable convergence determining class.

The class $\cA = \{x^k\}_{k = 0}^\infty$ of monic polynomials, in general, is not a convergence determining class but it is a convergence determining class for $\mu$, provided that the probability measure $\mu$ is determined by moments.

%\subsection{Almost sure convergence}

%
%\begin{definition}
%Let $\{\xi_n\}_{n \ge 1}$ and $\xi$ be random measures defined on the same probability space $(\Omega, \F, \P)$. We say that the sequence $\{\xi_n\}$ converges weakly to $\xi$ almost surely if 
%\[	
%	\P(\omega: \xi_n(\omega) \wto \xi(\omega) ) = 1.
%\] 
%\end{definition}

\begin{theorem}\label{thm:dc}
Let $\cA \subset C_b(\R)$ be a countable convergence determining class. 
Let $\{\xi_n\}_{n = 1}^\infty$ and $\xi$ be random probability measures defined on the same probability space $(\Omega, \F, \Prob)$. Then the following statements are equivalent: 
\begin{itemize}
	\item[\rm (i)] $\xi_n$ converges weakly to $\xi$ almost surely;
	\item[\rm (ii)]	for all $f \in C_b(\R)$, $\lf \xi_n, f \rt$  converges almost surely to $\lf \xi, f\rt$;
	\item[\rm (iii)]	for all $f \in \cA$, $\lf \xi_n,  f \rt$  converges almost surely to $\lf \xi, f\rt$.
\end{itemize}
\end{theorem}
\begin{proof}
	It is clear that (i) implies (ii), and (ii) implies (iii). We only need to show (iii) implies (i). Assume that (iii) holds. Let 
	\[
		A_f = \{\omega : \lim_{n \to \infty}\bra{\xi_n(\omega), f} = \bra{\xi(\omega), f}\}.
	\] 
Then $\Prob(A_f) = 1$ by the assumption. Let 
	\[
		A = \bigcap_{f \in \cA} A_f.
	\] 
Then $\Prob(A) = 1$ because $\cA$ is countable. Now for $\omega \in A$, 
\[
	\lf \xi_n(\omega), f \rt\to \lf \xi(\omega), f \rt\text{ for all $f \in \cA$},
\]
which implies that $\xi_n(\omega)$ converges weakly to $\xi(\omega)$, or $\rho(\xi_n(\omega), \xi(\omega)) \to 0$ by the definition of convergence determining class. 
\end{proof}

To deal with convergence in probability, we need the following result which is analogous to the usual case (\cite[Theorem~20.5]{Billingsley}). 
\begin{lemma}
	The sequence $\{\xi_n\}$ converges weakly to $\xi$ in probability, if and only if for any subsequence $\{\xi_{n_k} \}$, there is a further subsequence $\{\xi_{n_k'} \}$ which converges weakly to $\xi$ almost surely.
\end{lemma}

\begin{theorem}
Let $\cA \subset C_b(\R)$ be a countable convergence determining class. 
Let $\{\xi_n\}_{n = 1}^\infty$ and $\xi$ be random probability measures defined on the same probability space $(\Omega, \F, \Prob)$. Then the following statements are equivalent: 
\begin{itemize}
	\item[\rm (i)] $\xi_n$ converges weakly to $\xi$ in probability;
	\item[\rm (i)]	for all $f \in C_b(\R)$, $\lf \xi_n, f \rt$  converges in probability to $\lf \xi, f\rt$;
	\item[\rm (i)]	for all $f \in \cA$, $\lf \xi_n,  f \rt$  converges in probability to $\lf \xi, f\rt$.
\end{itemize}
\end{theorem}
\begin{proof}
	It is an easy consequence of the previous lemma and Theorem~\ref{thm:dc}. 
\end{proof}

\begin{corollary}
	The sequence of random probability measures $\{\xi_n\}$ converges weakly to $\xi$ almost surely (resp. in probability), if and only if $S_{\xi_n}(z)$ converges almost surely (resp. in probability) to $S_{\xi}(z)$ for all $z \in \C_+$, or for all $z \in D$, a dense subset of $\C_+$.
\end{corollary}

\begin{theorem}
Let $\xi = \mu$ be a non-random probability measure and let $\cA_\mu$ be a countable convergence determining class for $\mu$. Let $\{\xi_n\}_{n = 1}^\infty$ be a sequence of random probability measures.  Assume that for any $f \in \cA_\mu$, the sequence $\lf \xi_n, f\rt$ is well-defined and it converges to $\lf \mu, f\rt$ almost surely (resp.~in probability). Then the sequence $\{\xi_n\}$ converges weakly to $\mu$ almost surely (resp.~in probability).
\end{theorem}
\begin{proof}
	The proof is analogous to that of Theorem~\ref{thm:dc}. 
\end{proof}

\begin{corollary}\label{cor:moments-convergence}
	Assume that the probability measure $\mu$ is determined by its moments. Then the condition 
\[
	\lf \xi_n, x^k\rt\to \lf \mu, x^k \rt\text{ almost surely (resp.~in probability), for } k = 0,1, \dots,
\]
implies that $\{\xi_n\}$ converges weakly to $\mu$ almost surely (resp.~in probability). More generally, assume that the random probability measure $\xi$ is determined by its moments almost surely. Then the condition 
\[
	\lf \xi_n, x^k\rt\to \lf \xi, x^k \rt\text{ almost surely (resp.~in probability), for } k = 0,1, \dots,
\]
implies that $\{\xi_n\}$ converges weakly to $\xi$ almost surely (resp.~in probability). 
\end{corollary}

For convergence in distribution, the following result is analogous to the one for random measures (\cite[Theorem~4.2]{Kallenberg}).
\begin{theorem}
	The sequence $\{\xi_n\}$ converges to $\xi$ in distribution, if and only if $\lf \xi_n, f\rt$ converges in distribution to $\lf \xi, f \rt$ for all $f \in C_b(\R)$.
\end{theorem}

Let $\xi$ be a random probability measure. Then the mean of $\xi$, denoted by $\bar \xi$, is defined as 
\[
	\bar \xi(B) = \E[\xi(B)], \text{ for all } B \in \cB(\R).
\]
In can be also defined as a probability measure $\mu$ such that 
\[
	\lf \mu, f\rt = \E[\lf \xi, f\rt], \text{ for all } f \in C_b(\R).
\]
The above equation still holds for all non negative functions $f$, or even for all measurable functions $f$ such that $\E[\lf \xi, f \rt] < \infty$. It is clear that the almost sure convergence implies the convergence in probability which further implies the convergence in distribution. Suppose that $\{\xi_n\}$ converges in distribution to $\xi$. Then the sequence of mean measures $\{\bar \xi_n\}$ converge weakly to $\bar \xi$, the mean of $\xi$. Indeed, let $f$ be a bounded continuous function, $|f(x)| \le M$ for all $x \in \R$. Since $\bra{\xi_n, f}$ converges to $\bra{\xi, f}$ in distribution and $|\bra{\xi_n, f}| \le M$ for all $n$, it follows that 
\[
	\E[\bra{\xi_n, f}] \to \E[\bra{\xi, f}] \text{ as } n \to \infty,
\]
by the bounded convergence theorem. This means that $\bra{\bar \xi_n, f}$ converges to $\bra{\bar \xi, f}$ for any bounded continuous function $f$, which implies that $\bar \xi_n$ converges weakly to $\bar \xi$ as $n \to \infty$.

%\begin{acknowledgements}
%\end{acknowledgements}

%\bibliography{bib}   % name your BibTeX data base

%\begin{thebibliography}{}
%%
%% and use \bibitem to create references. Consult the Instructions
%% for authors for reference list style.
%\end{thebibliography}

\end{document}